\documentclass[11pt]{article}

\usepackage[round]{natbib}
\usepackage{setspace,amssymb,amsmath,bbm,graphicx, bm} %subfigure
\usepackage{amsthm}
\usepackage{subcaption}

\usepackage{fullpage}

\usepackage{color}

\theoremstyle{definition}
\newtheorem{assumption}{Assumption}
\newtheorem{theorem}{Theorem}
\newtheorem{lemma}{Lemma}

\newtheorem{corollary}{Corollary}
\newtheorem{condition}{Condition}

\def\indep{\begin{picture}(9,8)
         \put(0,0){\line(1,0){9}}
         \put(3,0){\line(0,1){8}}
         \put(6,0){\line(0,1){8}}
         \end{picture}
        }

\def\bT{{\bm{T}}}
\def\btau{{\bm{\tau}}}
\def\bX{{\bm{X}}}
\def\bW{{\bm{W}}}
\def\bU{{\bm{U}}}
\def\bV{{\bm{V}}}
\def\bI{{\bm{I}}}
\def\bbeta{{\bm{\beta}}}
\def\bgamma{{\bm{\gamma}}}
\def\bDelta{{\bm{\Delta}}}
\def\bS{{\bm{S}}}
\def\bzero{{\bm{0}}}
\def\bone{{\bm{1}}}
\def\bb{{\bm{b}}}
\def\ba{{\bm{a}}}
\def\bA{{\bm{A}}}
\def\bB{{\bm{B}}}
\def\bC{{\bm{C}}}
\def\bD{{\bm{D}}}
\def\bE{{\bm{E}}}

\def\bG{{\bm{G}}}
\def\bH{{\bm{H}}}
\def\bY{{\bm{Y}}}

\newcommand{\var}{\text{var}}

\newcommand{\cov}{\text{cov}}

\newcommand{\sumn}{\sum_{i=1}^n}

 \def\obs{\text{obs}}
 
 \def\S{\mathcal{S}}
 
  \def\CR{\textsc{CR}}
 
 \def\OLS{\textsc{OLS}}
 
  \def\RI{\textsc{RI}}
  
  \def\TSLS{\text{TSLS}}

\def\b1{\boldsymbol{1}}

\def\T{{\scriptscriptstyle  \texttt{T}}}
\def\d{\text{d}}

\begin{document}

\setlength{\baselineskip}{1.5\baselineskip}

\title{\bf \huge  Decomposing Treatment Effect Variation
\thanks{Peng Ding (Email: pengdingpku@berkeley.edu) is Assistant Professor, Department of Statistics, University of California, Berkeley. Avi Feller (Email: afeller@berkeley.edu) is Assistant Professor, Goldman School of Public Policy, University of California, Berkeley. Luke Miratrix (Email: lmiratrix@g.harvard.edu) is Assistant Professor, Harvard Graduate School of Education. We thank Alberto Abadie, Donald Rubin, participants at the Applied Statistics Seminar at the Harvard Institute of Quantitative Social Science, and colleagues at University of California, Berkeley and Harvard University for helpful comments. We gratefully acknowledge financial support from the Spencer Foundation through a grant entitled ``Using Emerging Methods with Existing Data from Multi-site Trials to Learn About and From Variation in Educational Program Effects,'' and from the Institute for Education Science (IES Grant \#R305D150040).
}
} 
\author{Peng Ding\\UC Berkeley \and Avi Feller\\UC Berkeley \and Luke Miratrix\\Harvard GSE}
\date{\today}
\maketitle

\begin{abstract}
Understanding and characterizing treatment effect variation in randomized experiments has become essential for going beyond the ``black box'' of the average treatment effect.  Nonetheless, traditional statistical approaches often ignore or assume away such variation.  
In the context of randomized experiments, this paper proposes a framework for decomposing overall treatment effect variation into a systematic component explained by observed covariates and a remaining idiosyncratic component.
Our framework is fully randomization-based, with estimates of treatment effect variation that are entirely justified by the randomization itself. Our framework can also account for noncompliance, which is an important practical complication. We make several contributions. First, we show that randomization-based estimates of systematic variation are very similar in form to estimates from fully-interacted linear regression and two stage least squares. Second, we use these estimators to develop an omnibus test for systematic treatment effect variation, both with and without noncompliance. Third, we propose an $R^2$-like measure of treatment effect variation explained by covariates and, when applicable, noncompliance. Finally, we assess these methods via simulation studies and apply them to the Head Start Impact Study, a large-scale randomized experiment.

\end{abstract}

\newpage 
\section{Introduction}

The analysis of randomized experiments has traditionally focused on the average treatment effect, often ignoring or assuming away treatment effect variation~\citep[e.g.,][]{neyman::1923, fisher::1935a, kempthorne::1952, rosenbaum::2002}. Today, understanding and characterizing treatment effect variation in randomized experiments has become essential for going beyond the ``black box'' of the average treatment effect. This is clear from the increasing number of papers on the topic in statistics and machine learning~\citep{hill2011bayesian, Athey:2015vu, wager2015estimation}, biostatistics~\citep{huang2012assessing, matsouaka2014evaluating}, education~\citep{raudenbush::2015}, economics~\citep{heckman::1997, crump::2008, djebbari::2008}, political science \citep{Green:2012kt, imai2013estimating}, and other areas.

This paper proposes a framework for decomposing overall treatment effect variation in a randomized experiment into a \textit{systematic component} that is explained by observed covariates, and an \textit{idiosyncratic component} that is not explained~\citep{heckman::1997, djebbari::2008}. In doing so, we make several key contributions. First, we take a fully randomization-based perspective \citep[see][]{rosenbaum::2002, imbens::2015}, and propose estimators that are entirely justified by the randomization itself. This is in contrast to much of the literature on randomization-based methods, where treatment effect variation is typically a nuisance~\citep[e.g.][]{rosenbaum1999reduced, rosenbaum2007confidence}. 
Similar to \citet{lin::2013}, we show that the resulting estimator is very similar in form to linear regression with interactions between the treatment indicator and covariates. Unlike with linear regression, however, the proposed estimator does not require any modeling assumptions on the marginal outcomes.

Second, we extend these methods from intention-to-treat (ITT) analysis to allow for noncompliance, proposing a randomized-based estimator for systematic treatment effect variation for the Local Average Treatment Effect (LATE) in the case of noncompliance~\citep{angrist::1996}. We show that this estimator is nearly identical to the two-stage least squares estimator with interactions between the treatment and covariates. We believe that this is a particularly novel contribution to the recent literature seeking to reconcile the randomization-based tradition in statistics and the linear model-based perspective more common in econometrics~\citep{abadie::2003, imbens2014instrumental, imbens::2015}.

Armed with these estimators, we turn to two practical tools for decomposing treatment effect variation. The first is an omnibus test for the presence of systematic treatment effect variation. While versions of this test have been proposed previously, largely in the context of linear models~\citep{cox1984interaction, crump::2008}, our proposed test is fully randomization-based and can also account for noncompliance. The second is to develop and bound an $R^2$-like measure of the fraction of treatment effect variation explained by covariates. This builds on previous versions proposed in the econometrics literature~\citep{heckman::1997, djebbari::2008}, again extending results to account for noncompliance. This approach is also closely related to the Oaxaca--Blinder decomposition in economics~\citep{oaxaca1973male, blinder1973wage}. See~\citet{angrist2013explaining} for a recent application that also addresses compliance. 
Finally, we apply these methods to the Head Start Impact Study, a large-scale randomized trial of Head Start, a federally funded preschool program \citep{puma2010head}.
We relegate the technical details and some further extensions to the online Supplementary Material.

\section{Framework for Treatment Effect Variation}
\label{sec::framework}

\subsection{Setup and notation}

Assume that we have $n$ units in an experiment. For unit $i$, let $\bX_i = (X_{1i}, \ldots, X_{Ki})^\T \in \mathbb{R}^K$ denote the vector of pretreatment covariates, with the constant 1 as its first component.
Let $T_i$ denote the treatment indicator with $1$ for treatment and $0$ for control. We use the potential outcomes framework~\citep{neyman::1923, rubin::1974} to define causal effects. Under the Stable Unit Treatment Value Assumption \citep{rubin::1980} that there is only one version of the treatment and no interference among units, we define $Y_i(1)$ and $Y_i(0)$ as the potential outcomes of unit $i$ under treatment and control, respectively. 
The observed outcome, $Y_i^\obs = T_iY_i(1) + (1-T_i)Y_i(0)$, is quite general and includes continuous, binary, and zero-inflated cases. 
On the difference scale, the individual treatment effect is $\tau_i = Y_i(1) - Y_i(0)$. 

Importantly, this is finite population inference in that we condition on the $n$ units at hand---the potential outcomes are fixed and pre-treatment.  This differs from super population inference in which some variables or residuals are assumed to be independent and identically distributed (iid) draws from some distribution. See, for example,~\citet{rosenbaum::2002}, \citet{imbens::2015} and \citet{li2017general}. Under the potential outcomes framework, $\{ Y_i(1), Y_i(0)\}_{i=1}^n$ are all fixed numbers; the randomness of any estimator comes from the assignment mechanism, which is the distribution of possible treatment assignments $\bT = (T_1, \ldots, T_n)^\T$.
Note that $\text{pr}\{  (T_1,\ldots, T_n)   = (t_1,\ldots, t_n)\} = \binom{n}{n_1}^{-1}$ if $\sum_{i=1}^n t_i = n_1$.

\subsection{Randomization inference for vector outcomes}
To set up our overall framework, we first generalize~\citet{neyman::1923}'s classic results to vector outcomes. 
We consider a completely randomized experiment, with $n_1$ units assigned to treatment and $n_0$ units assigned to control; in total we have $\binom{n}{n_1}$ possible randomizations. 
We are interested in estimating the finite population average treatment effect on a vector outcome $\bV  \in \mathbb{R}^K$:
$$
\btau_\bV = \frac{1}{n}  \sumn \left\{  \bV_i(1) - \bV_i(0) \right\},
$$
where $\bV_i(1)$ and $\bV_i(0)$ are the potential outcomes of $\bV$ for unit $i$. 
The Neyman-type unbiased estimator for $\btau_\bV$ is the difference between the sample mean vectors of the observed outcomes under treatment and control:
$$
\widehat{ \btau}_\bV = \bar{\bV}_1^{\obs} - \bar{\bV}_0^{\obs}
= \frac{1}{n_1}   \sumn T_i  \bV_i^\obs   - \frac{1}{n_0}  \sumn (1-T_i) \bV_i^\obs 
=\frac{1}{n_1}   \sumn T_i  \bV_i(1)   - \frac{1}{n_0}  \sumn (1-T_i) \bV_i(0) .
$$

The behavior of our estimator, and of our estimators for heterogeneity discussed later, revolve around covariances of vector outcomes.
For notation, let $\bA = \left\{ \bA_1, \ldots, \bA_n \right\}$ be a collection of $n$ vectors, with $\bar{\bA} = n^{-1}\sum_{i=1}^n A_i$ the vector mean, and define the covariance operator on $\bA$ as
$$
\mathcal{S}(\bA) =  \frac{1}{n-1}   \sumn  (  \bA_i - \bar{\bA}   )  (\bA_i - \bar{\bA})^\T , 
$$
which gives the covariance matrix of the $n$ vectors in $\bA$. For example, $\bA_i$ can be $\bV_i(1), \bV_i(0)$ or $\bV_i(1) - \bV_i(0).$

The following theorem, generalizing the results for scalar outcomes from \citet{neyman::1923}, demonstrates that $\widehat{ \btau}_\bV $ is unbiased and gives its covariance matrix.   

\begin{theorem}\label{thm::general-neyman}
Over all possible randomizations of a completely randomized experiment, $\widehat{ \btau}_\bV $ is unbiased for $\btau_\bV$, with $K \times K$ covariance matrix:
\begin{eqnarray}
\cov(\widehat{ \btau}_\bV) 
= \frac{    \mathcal{S} \{  \bV(1)  \} }{n_1} + 
 \frac{    \mathcal{S} \{ \bV(0) \} }{n_0}  - 
 \frac{    \mathcal{S} \{   \bV(1)  - \bV(0) \}  }{n}. 
 \label{eq::neyman-vector}
\end{eqnarray}
\end{theorem}

The diagonal elements of this matrix are the variances of the estimators of each component of $\btau_\bV$. 
The covariance matrix of $\widehat{ \btau}_\bV$ depends on the various covariances of the potential outcomes under treatment and control. 
In particular, the last term depends on the correlation between the potential outcomes $\bV(1)$ and $\bV(0)$, and therefore cannot be identified from the observed data. When the individual treatment effects are constant for all components of $\bV$, the last term in the above covariance matrix vanishes, because then $\mathcal{S} \{   \bV(1)  - \bV(0) \}  = \bzero_{K\times K}$. 
Under this assumption, we can unbiasedly estimate the sampling covariance matrix $\cov(\widehat{ \btau}_\bV) $ by replacing the covariances of the potential outcomes by the sample analogues:
$$
\widehat{\cov}(\widehat{\btau}_\bV) = 
\frac{\widehat{\mathcal{S}}_1(\bV^\obs)}{n_1} + \frac{\widehat{\mathcal{S}}_0(\bV^\obs)}{n_0},
$$ 
where
\begin{eqnarray}
\label{eq::sample-cov}
\widehat{\mathcal{S}}_t(  \bV^\obs ) = \frac{1}{n_t-1} \sumn I_{(T_i=t)}(\bV_i - \bar{\bV}_t^\obs) (\bV_i - \bar{\bV}_t^\obs)^\T \qquad (t = 0,1)
\end{eqnarray}
are the sample covariance matrices of $\bV^\obs$ in the treatment and control groups.
Without the constant treatment effect assumption, the covariance estimator $\widehat{\cov}(\widehat{\btau}_\bV)$ is conservative in the sense that the difference between the expectation of the variance estimator and the true variance is a non-negative definite matrix.
In particular, the diagonal terms of the expected estimator will all be larger than the truth.
Letting $K = 1$, the covariance matrices become simple variances, which recovers Neyman's original result.

Using the mathematical framework introduced in the Appendix and in~\citet{li2017general}, we can easily generalize Theorem \ref{thm::general-neyman} to more complicated experimental designs, e.g., cluster-randomized trials~\citep{middleton2016unbiased} and unbalanced $2^2$ split-plot designs \citep{zhao2016randomization}.

\subsection{Decomposing Treatment Effect Variation}

We now apply this general framework to treatment effect variation. We decompose the individual treatment effect, $\tau_i  $, via 
\begin{eqnarray}\label{eq::main-decomposition}
\tau_i = Y_i(1) - Y_i(0) =  \bX_i ^\T \bbeta + \varepsilon_i  , \quad (i=1, \ldots, n) 
\end{eqnarray}
with $\bbeta$ being the finite population linear regression coefficient of $\tau_i$ on $\bX_i$, defined by 
\begin{equation}
\bbeta = \arg \min_{\bb\in \mathbb{R}^K} \sum_{i=1}^n \left( \tau_i - \bX_i^\T \bb \right)^2 \label{eq::beta-minimization}.	
\end{equation} 
Following~\citet{heckman::1997} and~\citet{djebbari::2008}, we call $\delta_i = \bX_i^\T \bbeta $ the \textit{systematic treatment effect variation} explained by the observed covariates, $\bX_i$, and call $\varepsilon_i$  the \textit{idiosyncratic treatment effect variation} not explained by $\bX_i$.

More generally, we can view this decomposition in a regression-style framework.
Define 
$$
\bS_{xx} = \frac{1}{n} \sum_{i=1}^n \bX_i \bX_i^\T \in \mathbb{R}^{K\times K} ,  \quad 
\bS_{x\varepsilon}  = \frac{1}{n} \sum_{i=1}^n \varepsilon_i \bX_i \in \mathbb{R}^K ,  \quad 
\bS_{x\tau} = \frac{1}{n} \sum_{i=1}^n \tau_i \bX_i \in \mathbb{R}^K , 
$$
where $\bS_{xx}$ is non-degenerate, analogous to the usual full rank assumption in linear models. 
Also define
$$
\bS_{xt} = \frac{1}{n} \sumn \bX_i Y_i(t)  \in \mathbb{R}^{K}  , \quad (t=0,1) . 
$$ 
These are all finite population quantities, as in they are fixed pre-randomization values. 
The definition of $\bbeta$ gives $\bS_{x\varepsilon}=0$, i.e., $\varepsilon_i$ and $\bX_i$ have finite population covariance zero. Therefore, in the spirit of the agnostic regression framework~\citep[e.g.,][]{lin::2013}, the systematic component, $\delta_i = \bX_i ^\T \bbeta$, is a projection of $\tau_i$ onto the linear space spanned by $\bX_i$, and the idiosyncratic treatment effect, $\varepsilon_i$, is the corresponding residual. The linear projection applies to general outcomes, including the binary case.

Because of our finite population focus, if we observed all the potential outcomes we could immediately calculate all individual treatment effects and apply standard linear regression theory to (\ref{eq::main-decomposition}) and obtain $\bbeta$. 
In particular, the solution of~\eqref{eq::beta-minimization}, i.e. the ordinary least squares (OLS) solution from regressing $\tau$ on $\bX$, is
\begin{equation}
\bbeta = \bS_{xx}^{-1} \bS_{x\tau} = \bS_{xx}^{-1} \bS_{x1} -  \bS_{xx}^{-1} \bS_{x0} \equiv  \bgamma_1 - \bgamma_0 , \label{eq:population_beta}	
\end{equation}
where $\bgamma_1 = \bS_{xx}^{-1} \bS_{x1}$ and $\bgamma_0 = \bS_{xx}^{-1} \bS_{x0}$ are the corresponding finite population regression coefficients of the potential outcomes on the covariates. 
Let $
e_i(1) =  Y_i(1) - \bX_i^\T \bgamma_1$ and
$ 
e_i(0) = Y_i(0) - \bX_i^\T \bgamma_0
$
be the residual potential outcomes from the regression of $Y_i(t)$ onto $\bX$. 
Our idiosyncratic treatment variation is then the difference of residuals: $\varepsilon_i = e_i(1) - e_i(0)$.
In practice, we do not fully observe these components, but we can obtain unbiased or consistent estimates for them as we discuss below.

\section{Systematic treatment effect variation for the ITT}
\label{sec::systematic}

\subsection{Randomization-based estimator}
\label{sec::ri-hte}

We now turn to estimating $\bbeta$. As shown in~\eqref{eq:population_beta}, $\bbeta$ has three components.
The first term, $\bS_{xx}$, is fully observed as all the covariates are observed.
Our estimation then depends on the sample analogues of $\bS_{x1}$ and $\bS_{x0}$:
$$
\widehat{\bS}_{x1} = \frac{1}{n_1} \sum_{i=1}^n T_iY_i^\obs \bX_i \in \mathbb{R}^K ,\quad 
\widehat{\bS}_{x0} = \frac{1}{n_0} \sum_{i=1}^n (1-T_i) Y_i^\obs \bX_i  \in \mathbb{R}^K .
$$  
The $\widehat{\bS}_{xt}$'s capture how the observed potential outcomes correlate with the covariates. 
Plug these into~\eqref{eq:population_beta} to obtain an overall estimate of $\bbeta$.
The randomization of $\bT$ then justifies the following theorem.

\begin{theorem}
\label{thm::randomization-inference}
Under decomposition (\ref{eq::main-decomposition}), $\bS_{xx}^{-1}\widehat{\bS}_{x1}$ and $\bS_{xx}^{-1}\widehat{\bS}_{x0}$ are unbiased estimates of $\bgamma_1$ and $\bgamma_0$, respectively.
Therefore 
$$
\widehat{\bbeta}_{\RI} = \bS_{xx}^{-1}\widehat{\bS}_{x1} - \bS_{xx}^{-1}\widehat{\bS}_{x0},
$$
is an unbiased estimator for $\bbeta$ with covariance matrix
\begin{eqnarray}
\cov(\widehat{\bbeta}_{\RI}) = 
\bS_{xx}^{-1} 
\left[   \frac{  \S\{  Y(1)  \bX   \} }{ n_1} +\frac{  \S\{  Y(0)  \bX  \} } { n_0  } -  \frac{ \S( \tau \bX  ) } { n} \right ]
 \bS_{xx}^{-1}.
 \label{eq::covariance}
\end{eqnarray}
\end{theorem}

Here, for example, $\S\{ Y(0) \bX \}$ denotes the covariance operator on new unit-level variables $ Y_i(0)\bX_i \in \mathbb{R}^K$, made by scaling the $\bX_i$ vector of each unit by $Y_i(0)$,  similarly for $\S\{Y(1)\bX\}$ and $\S(  \tau\bX ) $.
This slight abuse of notation gives formulae less cluttered by subscripts and excessive annotation.
As with the vector version of Neyman's formula, the square root of the diagonal of $\cov(\widehat{\bbeta}_{\RI})$ gives the standard errors of $\beta$.

The covariance formula (\ref{eq::covariance}) generalizes the result of~\citet{neyman::1923} for the average treatment effect, reducing to Neyman's formula if $\bX_i=1$ for all units. 
%In Theorem \ref{thm::randomization-inference}, $\bS_{xx}$ is known for the finite population, rather than estimated. 
%
We can obtain a ``conservative'' estimate of $\cov(\widehat{\bbeta}_{\RI})$ by
$$
\widehat{\cov}(\widehat{\bbeta}_{\RI})  = 
\bS_{xx}^{-1} 
\left[   \frac{  \widehat{\S}_1(Y^\obs \bX ) }{ n_1} +\frac{  \widehat{\S}_0( Y^\obs \bX ) } { n_0  }  \right ]
 \bS_{xx}^{-1},
$$
recalling the definitions of the sample covariance operators $\widehat{\mathcal{S}}_1$ and $\widehat{\mathcal{S}}_0$ introduced in \eqref{eq::sample-cov}.
Similar to~\citet{neyman::1923}, this implicitly assumes $\S( \tau \bX  ) = \bzero$.
Under the assumption that $\varepsilon_i = 0$ for all units (i.e., no idiosyncratic variation whatsoever), we can instead use $\mathcal{S}( \widehat{\tau} \bX )$ with $\widehat{\tau} = \bX_i^\T \widehat{\bbeta}_\RI$ as a plug-in estimate for $\S( \tau  \bX  )$. This yields tighter standard errors based on the diagonal elements of the covariance matrix.

\paragraph{Finite population asymptotic analysis}
Theorem~\ref{thm::randomization-inference} holds for any finite sample.
To obtain confidence intervals and to conduct hypothesis testing as we describe below, we need to prove further that $ \widehat{\bbeta}_\RI$ is asymptotically normal with mean $\bbeta$ and covariance $\cov(\widehat{\bbeta}_{\RI} )$. Finite population asymptotic analysis, however, has a slightly different flavor from the usual super population approach. Formally, the finite asymptotic scheme embeds the finite population $\{ (\bX_i, Y_i(1), Y_i(0)) \}_{i=1}^n$ with size $n$ into a hypothetical sequence of finite populations with sizes approaching infinity. This effectively assumes that all the finite population quantities, for example, $\bS_{xx}$ and $\bbeta$, depend on $n$, although they are fixed numbers for a given finite population. Moreover, the sample quantities such as $\widehat{\bS}_{x1}$ and $\widehat{\bbeta}_{\RI}$ depend on $n$ as well, and are random quantities due to the randomization of $\bT$. For notational simplicity, we drop the index $n$ for all these quantities. Importantly, we must impose some regularity conditions on the hypothetical sequence of finite populations. Throughout the paper, we invoke the following conditions for asymptotic analysis, which are required for a form of the finite population central limit theorem discussed in~\citet[][Theorem 5]{li2017general}.

\begin{condition}
(i) Stable treatment proportions: $p_1=n_1/n$ and $p_0=n_0/n$ have positive limiting values; (ii) Stable means, variances and covariances: the finite population means, variances and covariances of the covariates and potential outcomes have finite limiting values; (iii) $\max_{1\leq i\leq n}||\bV_i - \bar{\bV}  ||^2_2/n \rightarrow 0$, where $\bV_i$ can be the covariate vector, the outcome, and the products of them. 
\end{condition}

Under these conditions, we can extend Theorem~\ref{thm::randomization-inference} to a sequence of finite populations:
\begin{eqnarray}   \label{eq::asynormal}
\sqrt{n} \left( \widehat{\bbeta}_\RI - \bbeta \right)  \stackrel{d}{\rightarrow} \mathcal{N}\left( \bzero, \lim_{n\rightarrow \infty}\left[  p_1 \S\{  Y(1)  \bX   \} + p_0 \S\{  Y(0)  \bX  \} - \S( \tau \bX  ) \right]  \right).
\end{eqnarray}
As a result, we can state that $\widehat{\bbeta}_{\RI}$ is approximately normal with mean $\bbeta$ and covariance matrix \eqref{eq::covariance}, which allows us to construct confidence intervals and hypothesis tests. In our theory below, we use this informal statement instead of \eqref{eq::asynormal} to avoid notational complexity. 

Conditions (i) and (ii) are natural. Condition (iii) holds if $\bV$ has more than two moments \citep{li2017general}. For bounded covariates and outcomes, (iii) is satisfied automatically.   
For more technical discussion of finite population causal inference, see \citet{ding::2014}, \citet{aronow::2014}, and \citet{middleton2016unbiased}; for regularity conditions of the finite population central limit theorems, see \citet{hajek::1960} and \citet{lehmann::1998}. A recent review is \citet{li2017general}.

\subsection{Regression with treatment-covariate interactions}
\label{sec:ols_int}
The results from randomization inference can shed light on the familiar case of linear regression with treatment-covariate interactions.
This classical approach assumes the model
\begin{equation}\label{eq::reg_int}
Y_i^\obs =  \bX_i ^\T  \bgamma + T_i   \bX_i^\T \bbeta   + u_i , \qquad (i=1, \ldots, n)
\end{equation}
where $\{  u_i \}_{i=1}^n$ are errors implicitly assumed to induce the randomness, and where $\bbeta$ models systematic treatment effect variation, as in~\eqref{eq::main-decomposition}.
Departing from much of the previous literature~\citep[e.g.,][]{cox1984interaction, berrington::2007, crump::2008}, we study the properties of the least squares estimator under complete randomization, without assuming that model \eqref{eq::reg_int} is correctly specified.
In particular, we do not assume any i.i.d. sampling; the assignment mechanism drives the distribution of the OLS estimator.

\begin{theorem}
\label{theorem::ols}
The OLS estimator for $\bbeta$ from fitting model~(\ref{eq::reg_int}) can be rewritten as
$$
\widehat{\bbeta}_{\OLS} = \widehat{\bS}_{xx,1}^{-1} \widehat{\bS}_{x1} - \widehat{\bS}_{xx,0}^{-1} \widehat{\bS}_{x0},
$$ 
where 
$$
\widehat{\bS}_{xx,t} = \frac{1}{n_t} \sumn I_{(T_i = t)}\bX_i \bX_i^\T, \quad (t = 0,1).
$$
Over all possible randomizations of $\bT$, $\widehat{\bS}_{xx,1}^{-1} \widehat{\bS}_{x1}$ and $\widehat{\bS}_{xx,0}^{-1} \widehat{\bS}_{x0}$ are consistent estimates of $\bgamma_1$ and $\bgamma_0$ respectively;
 $\widehat{\bbeta}_{\OLS}$ 
therefore follows an asymptotic normal distribution with mean $\bbeta$ and covariance matrix
\begin{eqnarray}
\cov(\widehat{\bbeta}_{\OLS}) = 
\bS_{xx}^{-1} 
\left[   \frac{  \S\{  e(1) \bX   \} }{ n_1} +\frac{  \S\{  e(0) \bX    \} } { n_0  } -  \frac{ \S( \varepsilon \bX  ) } { n} \right ]
 \bS_{xx}^{-1}.
 \label{eq::covariance-ols}
\end{eqnarray}
with $e_i(1), e_i(0)$, and $\varepsilon_i$ as defined after~\eqref{eq:population_beta}.
\end{theorem}

This estimate is simply the difference between $\widehat{\bgamma}_{1,\OLS} =  \widehat{\bS}_{xx,1}^{-1} \widehat{\bS}_{x1} $ and $\widehat{\bgamma}_{0,\OLS} =  \widehat{\bS}_{xx,0}^{-1} \widehat{\bS}_{x0}$, two OLS regressions run separately on each treatment arm. For treated units, define residual $\widehat{e}_i = Y_i^\obs - \bX_i^\T  \widehat{\bgamma}_{1,\OLS}$, and for control units, define residual $\widehat{e}_i = Y_i^\obs - \bX_i^\T  \widehat{\bgamma}_{0,\OLS}$. 
We can drop the unidentifiable term $\S( \varepsilon  \bX  )$, estimate $\S\{    e(1) \bX \} $ and $\S\{   e(0) \bX \} $ by their sample analogues, and conservatively estimate the asymptotic covariance matrix \eqref{eq::covariance-ols} by
\begin{eqnarray*}
\widehat{\cov}(\widehat{\bbeta}_\OLS) = 
\widehat{\bS}_{xx,1}^{-1} 
\left[
\frac{  \widehat{\S}_1(\widehat{e}   \bX  )   }{  n_1 } 
\right] \widehat{\bS}_{xx,1}^{-1} 
+\widehat{\bS}_{xx,0}^{-1} \left[ 
\frac{  \widehat{\S}_0(  \widehat{e} \bX )   }{  n_0 } 
 \right ]
\widehat{\bS}_{xx,0}^{-1} .
\end{eqnarray*}
This form of the sandwich variance estimator has the same probability limit as the Huber–White covariance estimator for linear model \eqref{eq::reg_int}~\citep{huber::1967, white::1980, lin::2013, angrist::2008}.

Importantly, $\widehat{\bbeta}_{\text{RI}}$ and $\widehat{\bbeta}_{\text{OLS}}$ are quite similar in form. 
In particular, $\widehat{\bbeta}_\RI$ uses the true $\bS_{xx}$ while $\widehat{\bbeta}_\OLS$ separately estimates the covariance matrix for each treatment arm, $ \widehat{\bS}_{xx,0}$ and $ \widehat{\bS}_{xx,1}$.  
The latter is effectively a ratio estimator.
Although this introduces some small bias (on the order of $1/n$), using the estimated $\widehat{\bS}_{xx,t}$ rather than true $\bS_{xx}$ can often lead to gains in precision, especially when covariates are strongly correlated with the potential outcomes.
In particular, the OLS estimator, by separately estimating the (known) $\bS_{xx}$ matrix for each treatment arm, can account for random imbalances in the covariates in both arms. 

The RI estimator, by comparison, has no adjustment whatsoever, and so cannot account for such random covariate imbalances.  
However, in Section \ref{sec::additional} below, and in the supplementary materials we introduce a different form of adjustment that uses covariates to make the estimates of the $\bS_{xt}$ more precise.
Depending on the structure of covariates, this estimator could be better or worse than OLS adjustment; we leave a thorough investigation of these trade-offs for future work.

Regardless, we again emphasize that we do {\em not} rely on classical OLS assumptions to justify the OLS estimator here. Rather, randomization (plus some mild regularity conditions for the finite sample asymptotics) justify our results.
For related discussion, see \citet{cochran::1977} on ratio estimators in surveys.

\subsection{Omnibus test for systematic variation}\label{sec::omnibus}
Finally, we can use these results to develop an omnibus test for the presence of any systematic treatment effect variation. The null hypothesis of no treatment effect variation explained by the observed covariates can be characterized by
$$
H_0(\bX): \bbeta_1 = 0,
$$
where $\bbeta_1$ contains all the components of $\bbeta$ except the first component corresponding to the intercept. Under $H_0(\bX)$, the individual treatment effects have no linear dependance on $\bX.$

We then construct a Wald-type test for $H_0(\bX)$ using an estimator $ \widehat{\bbeta}$ and its covariance estimator $\widehat{\text{cov}}(   \widehat{\bbeta} )$; it could be $\widehat{\bbeta}_\RI$ or $\widehat{\bbeta}_\OLS.$ Let $ \widehat{\bbeta}_{1}$ and $\widehat{\text{cov}}(   \widehat{\bbeta}_{1} )$ denote the sub-vector of $ \widehat{\bbeta}$ and sub-matrix of $\widehat{\text{cov}}(   \widehat{\bbeta} )$, corresponding to the non-intercept coordinates of $\bX$. We reject when
\begin{eqnarray}\label{eq::test_sys}
 \widehat{\bbeta}_{1}^\T 
\widehat{\text{cov}} ^{-1} (   \widehat{\bbeta}_{1} ) 
 \widehat{\bbeta}_{1} 
 > q_{K-1}(1-\alpha),
\end{eqnarray} 
where $q_{K-1}(1-\alpha)$ is the $1-\alpha$ quantile of the $\chi^2$ random variable with degrees of freedom $K-1$.

The test in \eqref{eq::test_sys} is nearly identical to the test proposed by~\citet{crump::2008}. They relax the parametric assumption by taking a ``sieve estimator'' approach, namely by using a quadratic form of the regression function, which allows for more flexible marginal distributions. Our approach differs in that we avoid modeling the marginal distributions entirely. 
If desired, we can add polynomials of $\bX$ (or other basis functions) into the model for $\delta$ to allow for more flexible systematic treatment effect variation, which could enhance power or model more complex relationships between the $bX$ and treatment impact.

\subsection{Additional considerations}\label{sec::additional}
In the Supplementary Material, we describe two additional points about systematic treatment effect variation that we briefly address here.
First, as mentioned above, we can use model-assisted estimation to improve the randomization-based estimator. 
In particular, improving estimation of $\widehat{\bS}_{xt}$ directly improves $\widehat{\bbeta}_{\text{RI}}$, as the $\widehat{\bS}_{xt}$ are the only random components. 
In particular, if we replace the standard sample estimator, $\widehat{\bS}_{xt}$, by a more efficient, model-assisted estimator, as in survey sampling~\citep{cochran::1977, sarndal2003model}, we can achieve meaningful precision gains in practice. 
More importantly, this setup allows researchers to assess systematic variation across one set of covariates while adjusting for another set.

Second, under the assumption of no idiosyncratic variation (i.e., $\varepsilon_i = 0$ for all $i$), we can obtain exact inference for $\bbeta$ by inverting a sequence of randomization-based tests. This complements previous work on randomization-based tests for the presence of idiosyncratic treatment effect variation~\citep{ding::2015jrssb}.

\section{Idiosyncratic treatment effect variation for ITT\label{sec:idio_var}}
After characterizing the systematic component of treatment effect variation, we now turn to characterizing the idiosyncratic component. Since this quantity is inherently unidentifiable, we propose sharp bounds on this component and a framework for sensitivity analysis. We then leverage these results to bound an $R^2$-like measure of the treatment effect variation explained by covariates.

\subsection{Bounds}\label{sec::bounds}
We first define the main quantities of interest:
$$
S_{\tau\tau} = \frac{1}{n} \sum_{i=1}^n (\tau_i -  \tau )^2 , \quad
S_{\delta\delta} =  \frac{1}{n} \sum_{i=1}^n(\delta_i -  \tau  )^2 , \quad 
S_{\varepsilon\varepsilon} = \frac{1}{n}  \sum_{i=1}^n \varepsilon_i^2  ,
$$
with  $\delta_i$ and $\varepsilon_i$ defined as in (\ref{eq::main-decomposition}). Then $S_{\tau\tau} = S_{\delta\delta} + S_{\varepsilon\varepsilon}$. We can immediately estimate $S_{\delta\delta}$ via the sample variance of $\{ \widehat{\delta}_i = \bX_i ^\T \widehat{\bbeta} \}_{i=1}^n $, where $\widehat{\bbeta}$ is a consistent estimator, e.g., $\widehat{\bbeta}_{\RI}$ or $\widehat{\bbeta}_{\OLS}$. 
However, the idiosyncratic variance, $S_{\varepsilon\varepsilon}$, is inherently unidentifiable because it depends on the joint distribution of potential outcomes.

We can, however, derive sharp bounds for $S_{\varepsilon\varepsilon}$. Let $F_1(y)$ and $F_0(y)$ be the empirical cumulative distribution functions of $\{ e_i(1)    \}_{i=1}^n$ and $\{ e_i(0)  \}_{i=1}^n$. 
Below we denote $e(t)$ as a random variable taking equal probabilities on $n$ values of $\{ e_i(t)\}_{i=1}^n.$
Based on the Fr\'echet--Hoeffding bounds \citep{hoeffding::1941, frechet::1951, nelsen::2007}, we can bound $S_{\varepsilon\varepsilon}$ as follows. 

\begin{theorem}
\label{thm::bounds_for_s_tau_tau}
$S_{\varepsilon\varepsilon}$ has sharp bounds
$
  \underline{S}_{\varepsilon\varepsilon} \leq S_{\varepsilon\varepsilon} \leq  \overline{S}_{\varepsilon\varepsilon},
$
where
\begin{eqnarray*}
\underline{S}_{\varepsilon\varepsilon} =
\int_0^1 \{  F_1^{-1}(u)   -  F_0^{-1}(u)     \}^2 du,
\quad 
\overline{S}_{\varepsilon\varepsilon} =
\int_0^1 \{   F_1^{-1}(u)   -  F_0^{-1}(1-u)     \}^2 du
\end{eqnarray*}
with $F^{-1}(u) = \inf\{ x: F(x) \geq u \}$ as the quantile function. 
The lower and upper bounds are attainable when $e(1)$ and $e(0)$ have the same ranks and opposite ranks, respectively.
\end{theorem}

The lower bound of $S_{\varepsilon\varepsilon}$ corresponds to a rank-preserving relationship between $e(1)$ and $e(0)$, and the upper bound of $S_{\varepsilon\varepsilon}$ corresponds to an anti-rank-preserving relationship between $e(1)$ and $e(0)$. Equivalently, they correspond to the cases where the Spearman rank correlation coefficients between $e(1)$ and $e(0)$ are $+1$ and $-1.$

In practice, we can often sharpen these bounds because we are unlikely to have negatively associated potential outcomes after adjusting for covariates. If we assume a nonnegative correlation between $e(1)$ and $e(0)$, we have the following corollary:

\begin{corollary}\label{coro::independent}
If the correlation between $e(1)$ and $e(0)$ is nonnegative, then the bounds for $S_{\varepsilon\varepsilon}$ become
$\underline{S}_{\varepsilon\varepsilon} \leq S_{\varepsilon\varepsilon} \leq V_1 + V_0$, where
$
V_t 
$
is the variance of $e(t)$ for $t=0,1.$
\end{corollary}

We can consistently estimate each quantity: $S_{\delta \delta}$ by the sample variance of $\bX_i^\T \widehat{\bbeta}$, $F_{e1}(y)$ and $F_{e0}(y)$ by $\widehat{F}_1(y)$ and $\widehat{F}_0(y)$, the empirical cumulative distribution functions of the residuals $\widehat{e}_i$ under treatment and control, and $V_1$ and $V_0$ by the variances of $\widehat{e}(1)$ and $\widehat{e}(0)$.

\paragraph{Variance of the overall ITT estimator.} We can use these results to obtain sharper bounds on the variance of~\citet{neyman::1923}'s estimate of overall ITT, $\widehat{\tau} = n_1^{-1}\sumn T_iY_i^\obs - n_0^{-1}\sumn (1-T_i)Y_i^\obs$, extending previous work by~\citet{heckman::1997} and~\citet{aronow::2014}. See also~\citet{fogarty2016regression}. Applying the results in Section~\ref{sec::framework} for scalar outcomes, we have the following variance for the difference-in-means estimator,
$$
\var(\widehat{\tau}) = \frac{S_{11}}{n_1} + \frac{S_{00}}{n_0} - \left(  \frac{S_{\delta\delta} }{n}  + \frac{S_{\varepsilon\varepsilon} }{n} \right),
$$
where $S_{\tau\tau} = S_{\delta\delta} + S_{\varepsilon\varepsilon}$. As we discuss above,~\citet{neyman::1923} proposed a lower bound for the overall $\var(\widehat{\tau})$ under the assumption of a constant treatment effect, $S_{\tau\tau}=0$. More recently,~\citet{aronow::2014} instead proposed to bound $S_{\tau\tau}$ via Fr\'echet--Hoeffding bounds. We can modestly improve these results by applying Fr\'echet--Hoeffding bounds for $S_{\varepsilon\varepsilon}$ alone rather than for $S_{\tau\tau} = S_{\delta\delta} + S_{\varepsilon\varepsilon}$. So long as $S_{\delta\delta} > 0$, this yields strictly tighter bounds on $\var(\widehat{\tau})$ than the corresponding bounds that do not incorporate covariate information. In turn, this gives a tighter estimate of the standard error for the same difference-in-means estimator, $\widehat{\tau}$. 

\paragraph{A variance ratio test.} Finally, while the relationship between $e(0)$ and $e(1)$ is inherently unidentifiable, there is some information in the data about the relationship between $\varepsilon_i$, the individual-level idiosyncratic treatment effect, and $Y_i(0)$, the control potential outcome. In particular,~\citet{raudenbush::2015} noted that if the variance of the treatment potential outcomes is smaller than the variance of the control potential outcomes, then the treatment effect must be negatively associated with the control potential outcomes. In the Supplementary Material, we extend this result to incorporate covariates and propose a formal test.

\subsection{Sensitivity analysis}\label{sec::sensitivity-analysis}
Going beyond worst-case bounds, we can assess the sensitivity of our estimate of $S_{\varepsilon\varepsilon}$ to different assumptions of the dependence between potential outcomes. Using the probability integral transformation, we represent the residual potential outcomes as 
$$
e (1) = F^{-1}_1(U_1), \quad e (0) =  F^{-1}_0(U_0),\quad U_1, U_0 \sim \text{Uniform}(0,1),
$$
Therefore, the dependence of the potential outcomes is determined by the dependence of the uniform random variables $U_1$ and $U_0$, which are the standardized ranks of the potential outcomes. When $U_1 = U_0$, $S_{\varepsilon\varepsilon}$ attains the lower bound $\underline{S}_{\varepsilon\varepsilon}$; when $U_1 = 1 - U_0$, $S_{\varepsilon\varepsilon}$ attains the upper bound $\overline{S}_{\varepsilon\varepsilon}$; when $U_1\indep U_0$, $S_{\varepsilon\varepsilon}$ attains the improved upper bound $V_1 + V_0$. 

Rather than simply examine extreme scenarios of $S_{\varepsilon\varepsilon}$, we can instead represent $U_1$ as a mixture of $U_0$ and another independent uniform random variable $V_0:$
\begin{equation}
U_1 \sim \rho U_0 + (1- \rho) V_0,\quad U_0, V_0 \stackrel{\text{i.i.d.}}{\sim} \text{Uniform}(0,1),	 \label{eq::copula}
\end{equation}
which the sensitivity parameter $\rho$ captures the association between $U_1$ and $U_0$. An immediate interpretation of $\rho$ is the proportion of rank preserved units, with the other $1-\rho$ as the proportion of units with independent treatment and control  residual outcomes.
When $\rho = 0$, $U_1\indep U_0$, and the residual potential outcomes are independent; when $\rho = 1$, $U_1 = U_0$, and the residual potential outcomes have the same ranks. The values between $(0,1)$ corresponds to positive rank correlation but not full rank preservation. 
Note that the representation of the joint distribution is not unique, because we can choose any copula as a joint distribution of $(U_1,U_0)$ \citep{nelsen::2007}.
We choose the above representation and notation $\rho$ for the following theorem.

\begin{theorem}
\label{thm::sensitivity analysis}
If Equation~\ref{eq::copula} holds, then $\rho$ is Spearman's rank correlation coefficient between $e(1)$ and $e(0)$. Furthermore, $S_{\varepsilon\varepsilon}$ is a linear function of $\rho$:
$$
S_{\varepsilon\varepsilon}(\rho) = \rho \underline{S}_{\varepsilon\varepsilon} + (1-\rho) (V_1 + V_0).
$$
\end{theorem}

We cannot extract any information about $\rho$ from the data. We therefore treat $\rho$ as a sensitivity parameter, choose a plausible range of $\rho$, and obtain corresponding values for $S_{\varepsilon\varepsilon}$.

\subsection{Fraction of treatment effect variation explained}
A natural question is the relative magnitudes of $S_{\delta\delta}$ and $S_{\varepsilon\varepsilon}$~\citep{djebbari::2008}. Continuing the regression analogy, this is an $R^2$-like measure for the proportion of total treatment effect variation explained by the systematic component:
$$
R^2_\tau = \frac{S_{\delta\delta}}{S_{\tau\tau}}  
= \frac{     S_{\delta \delta}   }{  S_{\delta\delta} + S_{\varepsilon\varepsilon}},
$$
which is the ratio between the finite population variances of $\delta$ and $\tau.$ As above, we can directly estimate $S_{\delta\delta}$ but must bound $S_{\varepsilon\varepsilon}$. Applying Theorem~\ref{thm::bounds_for_s_tau_tau}, we obtain the following bounds on $R^2_\tau$.

\begin{corollary}\label{coro::bounds-r}
The sharp bounds on $R^2_\tau$ are
$$
\frac{     S_{\delta \delta}   }{  S_{\delta\delta} + \overline{S}_{\varepsilon\varepsilon}} \leq 
R^2_\tau
   \leq 
   \frac{     S_{\delta \delta}   }{  S_{\delta\delta} + \underline{S}_{\varepsilon\varepsilon}}.
$$
If we further assume that the correlation between $e(1)$ and $e(0)$ is nonnegative, the sharp bounds  on $R^2_\tau$ are
$$
\frac{     S_{\delta \delta}   }{  S_{\delta\delta} + V_1 + V_0} \leq 
R^2_\tau
   \leq 
   \frac{     S_{\delta \delta}   }{  S_{\delta\delta} + \underline{S}_{\varepsilon\varepsilon}}.
$$
\end{corollary}
We estimate these bounds via plug-in estimates. Note that~\citet{djebbari::2008} explore a similar quantity by using a permutation approach to approximate the Fr\'echet--Hoeffding upper and lower bounds.
Finally, we can use the sensitivity results for ${S}_{\varepsilon\varepsilon}$, with values of $\rho \in [0,1]$:
$$
R^2_\tau(\rho) = \frac{     S_{\delta \delta}   }{  S_{\delta\delta} + {S}_{\varepsilon\varepsilon}(\rho)}.
$$

\section{Noncompliance}
\label{sec::noncompliance}

\subsection{Setup}
We now extend our results to allow for noncompliance. Let $T$ be the indicator of treatment assigned, $D$ be the indicator of treatment received, $Y$ be outcome of interest, and $\bX$ be pretreatment covariates. Under the Stable Unit Treatment Value Assumption, we define $D_i(t)$ and $Y_i(t)$ as the potential outcomes for unit $i$ under treatment assignment $t.$ Following~\citet{angrist::1996} and~\citet{frangakis::2002}, we can classify units into four compliance types based on the joint values of $D_i(1)$ and $D_i(0)$:
	$$
	U_i = \begin{cases} \text{Always Taker}~(a) & \mbox { if } D_i(1) = 1, D_i(0) = 1, \\
	 \text{Never Taker}~(n) & \mbox { if } D_i(1) = 0, D_i(0) = 0, \\
	 \text{Complier}~(c) & \mbox { if } D_i(1) = 1, D_i(0) = 0, \\
	 \text{Defier}~(d) & \mbox { if } D_i(1) = 0, D_i(0) = 1 .	
 \end{cases}
 $$
Denote $n_i$ and $\pi_u$ by the number and proportion of compliance types $\pi_u$ of stratum $U=u$ for $u=a,n,c,d$.

Throughout our discussion, we invoke the following assumptions which are commonly used for analyzing randomized experiments with noncompliance. 
\begin{assumption}\label{ass::iv}
(i) Monotonicity: $D_i(1) \geq D_i(0)$; (ii) Exclusion restrictions for Always Takers and Never Takers: $Y_i(1)= Y_i(0)$ for all units with $D_i(1) = D_i(0)$; (iii) Strong instrument: $\pi_c >  C_0 >0$, where $C_0$ is a positive constant independent of the sample size. 
\end{assumption}

Monotonicity rules out the existence of Defiers, i.e., $\pi_d = 0$. Under monotonicity, we can estimate the proportion $\pi_u$ using the observed counts of units classified by $T$ and $D$: let $n_{td} = \#\{ i:  T_i = t, D_i = d  \}$, and then $\widehat{\pi}_n= n_{10}/n_1$, $\widehat{\pi}_a = n_{01}/n_0$,  and $\widehat{\pi}_c = n_{11}/n_1 - n_{01}/n_0$. The exclusion restrictions assume that treatment assignment has no effect on the outcome for Always Takers and Never Takers. As a result, treatment effect variation is trivially zero for Always Takers and Never Takers. Note that this is the unit-level exclusion restriction imposed in~\citet{angrist::1996}. This can be relaxed in other settings; for example, we could assume the impact of randomization for these groups is zero on average~\citep[see][]{imbens::2015}. Finally, to avoid technical complexity, we rule out the weak instrument case \citep{bound1995problems, staiger1997instrumental}, i.e., $\pi_c$ is within a small neighborhood of $0$ with radius shrinking to $0.$

We are interested in treatment effect variation among Compliers, which motivates the following decomposition:
\begin{eqnarray}
\label{eq::decomposition-noncompliance}
\tau_ i = Y_i(1) - Y_i(0) = \left\{
                \begin{array}{ll}
                  0, &\text{if }  U_i = a \text{ or } n,\\
                  \bX_i^\T \bbeta_c + \varepsilon_i, &\text{if } U_i=c,
                \end{array}
              \right.
\end{eqnarray}
where $\bbeta_c$ is the regression coefficient of $\tau_i$ on $\bX_i$ among Compliers, analogous to~\eqref{eq::main-decomposition}.

\subsection{Systematic treatment effect variation among Compliers}
\label{sec::noncompliance-systematic}
\subsubsection{Randomization inference}
We now extend the results of Section~\ref{sec::systematic} to estimate systematic treatment effect variation among Compliers. Define 
$$
\bS_{xx,u} = \frac{1}{n_u} \sum_{i=1}^n I_{(U_i = u)}  \bX_i \bX_i^\T,\quad
\bS_{xt,u} = \frac{1}{ n_u  } \sum_{i=1}^n  I_{(U_i = u)} Y_i(t) \bX_i , \quad (t=0,1; u = a, c, n) .
$$
%for $u = a, c, n$.
Then, analogous to (\ref{eq:population_beta}),  
\begin{eqnarray}
\label{eq::population_beta_noncompliance}
\bbeta_c = \bS_{xx,c}^{-1}  ( \bS_{x1, c} - \bS_{x0, c} )  =  \bS_{xx,c}^{-1}  \bS_{x1, c} - \bS_{xx,c}^{-1} \bS_{x0, c}  \equiv \bgamma_{1c} - \bgamma_{0c} ,
\end{eqnarray}
where 
$$
\bgamma_{1c} = \bS_{xx,c}^{-1}  \bS_{x1, c},\quad
\bgamma_{0c} = \bS_{xx,c}^{-1}  \bS_{x0, c}
$$
are the linear regression coefficients of $Y(1)$ and $Y(0)$ on covariates among Compliers.

Unlike in the ITT case, we cannot estimate these quantities directly. Instead, following standard results from noncompliance~\citep[e.g.,][]{angrist::1996, abadie::2003, angrist::2008}, we use estimates from observed subgroups to estimate the desired quantities of interest. Define sample moments:
\begin{eqnarray}
\widehat{\bS}_{xx,td} = \frac{1}{n_t} \sum_{i=1}^n I_{(T_i=t)} I_{(D_i=d)} \bX_i \bX_i^\T ,\quad
\widehat{\bS}_{xt, td} = \frac{1}{n_t} \sum_{i=1}^n I_{(T_i=t)} I_{(D_i=d)} Y_i^\obs \bX_i  \quad (t,d=0,1).
\label{eq::sample-quantities}
\end{eqnarray}
The following theorem connects these quantities with the finite population quantities in (\ref{eq::population_beta_noncompliance}). 

\begin{theorem}
\label{thm::randomization-inference-noncompliance}
Over all possible randomizations of a completely randomized experiment, both $\widehat{\bS}_{xx}(1) = \widehat{\bS}_{xx,11} - \widehat{\bS}_{xx,01}  $ and $\widehat{\bS}_{xx}(0) =  \widehat{\bS}_{xx,00} - \widehat{\bS}_{xx,10} $ are unbiased for $ \pi_c \bS_{xx,c}$, and
\begin{eqnarray}
E( \widehat{\bS}_{x1,11} - \widehat{\bS}_{x0,01}     )  =   \pi_c \bS_{x1,c},\quad
E( \widehat{\bS}_{x0,00} - \widehat{\bS}_{x1,10}     )  =   \pi_c \bS_{x0,c} .
\label{eq::unbiased-xy}
\end{eqnarray}
\end{theorem}

This theorem shows that we can obtain unbiased estimates for all terms in~\eqref{eq::population_beta_noncompliance}. 
The following corollary shows that we can then obtain consistent estimates for $\bgamma_{1c}$, $\bgamma_{0c}$, and $\bbeta_c$, recalling that in the asymptotic analysis, we need to embed $\{ \bX_i, Y_i(1), Y_i(0), D_i(1), D_i(0) \}_{i=1}^n$ into a hypothetical sequence of finite populations under Condition 1.

\begin{corollary}\label{coro::ri-nomcompliance}
$
\widehat{\bgamma}_{1c,\RI} = \widehat{\bS} ^{-1}_{xx}(1)   ( \widehat{\bS}_{x1,11} - \widehat{\bS}_{x0,01}     )
$
and
$
\widehat{\bgamma}_{0c, \RI} = \widehat{\bS} ^{-1}_{xx}(0)  ( \widehat{\bS}_{x0,00} - \widehat{\bS}_{x1,10}        )
$
are consistent for $\bgamma_{1c}$ and $\bgamma_{0c}$. 
Furthermore, $\widehat{\bbeta}_{c,\RI} = \widehat{\bgamma}_{1c,\RI} - \widehat{\bgamma}_{0c,\RI}$ is consistent for $\bbeta_c$ and follows an asymptotic normal distribution with covariance matrix
\begin{eqnarray}\label{eq::asym-cov-noncompliance}
\cov(\widehat{\bbeta}_{c,\RI}) = 
( \pi_c \bS_{xx,c} )^{-1}
\left[  
 \frac{  \S\{  e'(1)  \bX   \} }{ n_1} +\frac{  \S\{   e'(0) \bX   \} } { n_0  } -  \frac{ \S(  \varepsilon \bX ) } { n}
\right]
( \pi_c \bS_{xx,c} )^{-1},
\end{eqnarray}
where we define the residual potential outcomes to be:
\begin{eqnarray}\label{eq::potential-residual-noncompliance}
e'_i(1) = \left\{
                \begin{array}{l}
                  Y_i(1) - \bX_i^\T \bgamma_{1c},\\
                  Y_i(1) - \bX_i^\T \bgamma_{0c},\\
                  Y_i(1) - \bX_i^\T \bgamma_{1c},
                \end{array}
              \right.
\quad
e'_i(0) = \left\{
                \begin{array}{lll}
                  Y_i(0) - \bX_i^\T \bgamma_{1c},&& U_i=a,\\
                  Y_i(0) - \bX_i^\T \bgamma_{0c},&&U_i=n,\\
                  Y_i(0) - \bX_i^\T \bgamma_{0c},&& U_i=c.
                \end{array}
             \right.
\end{eqnarray}
\end{corollary} 

The idiosyncratic variation is $\varepsilon_i = e_i'(1) - e_i'(0)$ for unit $i$, with $\varepsilon_i = 0$ for Never Takers and Always Takers, and with $\varepsilon_i$ for Compliers as in~\eqref{eq::decomposition-noncompliance}. 
The two sets of residuals are not formed from a regression on all units, but instead the population regression on Compliers alone. As in the ITT case, we can estimate $\S\{  e'(1)  \bX   \}$ and $\S\{  e'(0) \bX   \}$ using their sample analogues; $\S(\varepsilon  \bX  )$, however, is unidentifiable. For units with $D_i = 1$, we define the residual $\widehat{e}'_i = Y_i^\obs - \bX_i^\T \widehat{\bgamma}_{c1,\RI}$, and for units with $D_i=0$, we define the residual $\widehat{e}'_i = Y_i^\obs - \bX_i^\T \widehat{\bgamma}_{c0,\RI}$. Therefore, we can obtain a conservative estimate for the asymptotic covariance \eqref{eq::asym-cov-noncompliance} by the following sandwich form:
$$
\widehat{\cov}(\widehat{\bbeta}_{c,\RI}) = 
\widehat{\bS} ^{-1}_{xx}(1) 
\left[    \frac{  \widehat{\S}_1(\widehat{e}'   \bX  )   }{n_1}         \right]
\widehat{\bS} ^{-1}_{xx}(1) 
+
\widehat{\bS} ^{-1}_{xx}(0) 
\left[   \frac{  \widehat{\S}_0(\widehat{e}'  \bX  )    }{n_0}         \right]
\widehat{\bS} ^{-1}_{xx}(0) .
$$ 
As with the ITT analog, so long as we have Assumption \ref{ass::iv}, randomization itself fully justifies the theorem and estimators without relying on a model of the observed outcomes.

\subsubsection{Two-Stage Least Squares}
We now turn to the standard two-stage least squares (TSLS) setting in econometrics~\citep[e.g.,][]{angrist::2008}. First, we impose a linear regression model with treatment-covariate interactions:
\begin{eqnarray*}
Y_i^\obs = \bX_i^\T   \bgamma + D_i \bX_i^\T  \bbeta + u_i \quad (i=1,\ldots, n). \label{eq::regression-noncompliance}
\end{eqnarray*}
Here, the randomness of the observed outcome comes from the randomness of $D_i$ and $u_i$. In the language of econometrics, the treatment received is ``endogenous,'' i.e., $D_i$ and the error term $u_i$ are assumed to be correlated; we therefore use $T_i$ as an instrument for $D_i$. The TSLS estimates $(\widehat{\bgamma}_{\TSLS}, \widehat{\bbeta}_{\TSLS})$ are the solutions to the following estimating equations:
\begin{equation}
\label{eq::TSLS}
n^{-1} \sum_{i=1}^n 
\begin{pmatrix}
\bX_i\\
T_i \bX_i
\end{pmatrix}
(Y_i^\obs -  \bX_i^\T   \widehat{\bgamma}_\TSLS -  D_i \bX_i^\T  \widehat{\bbeta}_\TSLS   ) = 0.
\end{equation}
This approach is based on $M$-estimation, though there are many other ways to formalize the TSLS estimator~\citep[e.g.,][]{imbens2014instrumental}. 
The following theorem shows that the fully-interacted TSLS estimator $\widehat{\bbeta}_\TSLS$ is consistent for $\bbeta_c$ across randomizations.
\begin{theorem}
\label{thm::TSLS}
Over all randomizations, the TSLS estimator $\widehat{\bbeta}_\TSLS$ follows an asymptotic normal distribution with mean $\bbeta_c$ and covariance matrix
\begin{eqnarray*}
( \pi_c \bS_{xx,c} )^{-1}
\left[
 \frac{  \S\{  e''(1)  \bX    \} }{ n_1} +\frac{  \S\{  e''(0) \bX  \} } { n_0  } -  \frac{ \S(\varepsilon  \bX  ) } { n}
\right]
( \pi_c \bS_{xx,c} )^{-1},
\end{eqnarray*}
where the residual potential outcomes are defined as
\begin{eqnarray*}
e''_i(1) = \left\{
                \begin{array}{l}
                  Y_i(1) - \bX_i^\T (\bgamma_{\infty} + \bbeta_c ) ,\\
                  Y_i(1) - \bX_i^\T \bgamma_{\infty},\\
                  Y_i(1) - \bX_i^\T  (\bgamma_{\infty} + \bbeta_c ),
                \end{array}
              \right.
\quad
e''_i(0) = \left\{
                \begin{array}{lll}
                  Y_i(0) - \bX_i^\T (\bgamma_{\infty} + \bbeta_c ),&& U_i=a,\\
                  Y_i(0) - \bX_i^\T \bgamma_{\infty},&& U_i=n\\
                  Y_i(0) - \bX_i^\T \bgamma_{\infty},&& U_i=c,
                \end{array}
              \right.              
\end{eqnarray*}
where $\bgamma_\infty$ is the probability limit of the TSLS regression coefficient, $\widehat{\bgamma}_\TSLS$, and the idiosyncratic treatment effect is $\varepsilon_i \equiv e_i''(1) - e_i''(0).$
\end{theorem}

For variance estimation, define the residual as $\widehat{e}''_i = Y_i^\obs - \bX_i^\T  (\widehat{\bgamma}_\TSLS+ \widehat{\bbeta}_\TSLS )$ for units with $D_i = 1$ and $\widehat{e}''_i = Y_i^\obs - \bX_i^\T  \widehat{\bgamma}_\TSLS$ for units with $D_i=0$. 
We can then use the following sandwich variance estimator
$$
\widehat{\cov}(\widehat{\bbeta}_{\TSLS}) = 
\widehat{\bS} ^{-1}_{xx}(1) 
\left[    \frac{  \widehat{\S}_1( \widehat{e}''  \bX  )   }{n_1}         \right]
\widehat{\bS} ^{-1}_{xx}(1) 
+
\widehat{\bS} ^{-1}_{xx}(0) 
\left[   \frac{  \widehat{\S}_0(  \widehat{e}'' \bX  )    }{n_0}         \right]
\widehat{\bS} ^{-1}_{xx}(0) ,
$$
which has the same probability limit as the Huber--White covariance estimator for $\widehat{\bbeta}_\TSLS.$
Therefore, the randomization itself effectively justifies the use of TSLS for estimating systematic treatment effect variation among Compliers, extending our ITT results. 

%Because $\widehat{\bbeta}_{c,\RI}$ and $\widehat{\bbeta}_\TSLS$ are inherently ratio estimators, there is no guarantee of unbiasedness across randomizations. 

Finally, while $\widehat{\bbeta}_\TSLS$ is a consistent estimator for $\bbeta_c$, $\widehat{\bgamma}_\TSLS$ is not, in general, a consistent estimator for $\bgamma_{c0}$; that is, $\bgamma_\infty \neq \bgamma_{c0}$. Instead, $\widehat{\bgamma}_\TSLS$ converges to $\bgamma_\infty = \bS_{xx}^{-1}\bS_{x0} - \pi_a \bS_{xx}^{-1} \bS_{xx,a}\bbeta_c$.
In the special case of one-sided noncompliance (i.e., $\pi_a = 0$), $\bgamma_{\infty} = \bgamma_0 = \bS_{xx}^{-1}\bS_{x0}$,
the population OLS regression coefficient, among all Compliers and Never Takers, of $Y(0)$ on covariates.

\subsubsection{Omnibus test for systematic treatment effect variation among Compliers} 
With point estimate $\widehat{\bbeta}$ and covariance estimate $\widehat{\cov}(\widehat{\bbeta})$ for $\bbeta_{c}$, we can use the same Wald-type $\chi^2$ test as in \eqref{eq::test_sys} for the presence of systematic treatment effect variation among Compliers. Here, the estimator can be either  randomization-based $\widehat{\bbeta}_{c,\RI}$ or TSLS estimator $\widehat{\bbeta}_\TSLS$; the degrees of freedom are the same, $K-1$. Unlike in the ITT case, we are not aware of existing tests for systematic treatment effect variation among Compliers.

\subsection{Idiosyncratic treatment effect variation with noncompliance}\label{sec::decomposition-with-noncompliance}
\subsubsection{Bounding idiosyncratic variation}
We now turn to decomposing the overall treatment effect in the presence of noncompliance. In this setting, we have three sources of treatment effect variation: (i) systematic treatment effect variation among Compliers, (ii) idiosyncratic treatment effect variation among Compliers, and (iii) treatment effect variation due to noncompliance.

First, recall that total treatment effect variation is $S_{\tau\tau} =  \sum_{i=1}^n  (\tau_i  - \tau)^2 / n$. We can define a similar quantity among Compliers:
\begin{eqnarray*}
S_{\tau\tau,c} = \frac{1}{n_c}   \sum_{i=1}^n I_{(U_i = c) } (\tau_i -  \tau_c)^2 .
\end{eqnarray*}
As in Section~\ref{sec:idio_var}, we can decompose this variation into systematic and idiosyncratic treatment effect variation for Compliers, respectively: 
\begin{eqnarray*}
S_{\delta\delta, c} = \frac{1}{n_c} \sum_{i=1}^n I_{(U_i = c) } (\delta_i - \tau_c)^2, \qquad S_{\varepsilon\varepsilon, c} = \frac{1}{n_c} \sum_{i=1}^n I_{(U_i = c) }  \varepsilon_i^2.
\end{eqnarray*}
Because treatment effects for Never Takers and Always Takers are zero, there is no treatment effect variation for these units. 
The component of treatment effect variation due to compliance status is
\begin{eqnarray*}
S_{\tau\tau, U} = \sum_{u=c,a,n}  \pi_u  ( \tau_u -  \tau)^2 .
\end{eqnarray*}
Using $ \tau_a= \tau_n =0$ and $ \tau = \pi_c  \tau_c$ due to the exclusion restrictions, we have the following theorem summarizing the relationships among the above components.

\begin{theorem}
\label{thm::decomposition-noncompliance}
$
S_{\tau\tau}  = \pi_c  S_{\tau\tau,c} + S_{\tau\tau, U},
$ 
$
S_{\tau\tau, c} = S_{\delta \delta, c} + S_{\varepsilon\varepsilon,c},
$
and
$
S_{\tau\tau,U} = \pi_c(1-\pi_c) \tau_c^2.
$
\end{theorem}
In words, total treatment effect variation has three parts: (i) systematic treatment effect variation among Compliers, $\pi_c S_{\delta\delta,c}$; (ii) idiosyncratic treatment effect variation among Compliers,  $\pi_c S_{\varepsilon\varepsilon,c}$; (iii)  treatment effect variation due to noncompliance, $S_{\tau\tau,U}$.

As in the ITT case, even though $S_{\varepsilon\varepsilon,c}$ is not identifiable, we can derive bounds in terms of the marginal distributions of the residuals, $\{ e_i'(1) =  Y_i(1) - \bX_i^\T \bgamma_{1c}: U_i=c , i=1,\ldots,n \}$ and $\{ e_i' (0)  = Y_i(0) - \bX_i^\T \bgamma_{0c}: U_i=c ,i=1,\ldots,n \}$, denoted by $F_{1c}(y)$ and $F_{0c}(y)$, and with marginal variances, $V_{1c}$ and $V_{0c}.$ 
Once we estimate these quantities, we can plug them in to Theorem~\ref{thm::bounds_for_s_tau_tau} and Corrolary~\ref{coro::independent} to get our bounds.
As compliance status is only partially observed, we have to estimate these quantities by differencing observed distributions; we defer this and some other technical details to the Supplementary Material.

%We show this with the following theorem. 
%
%
%
%\begin{corollary}
%\label{coro::noncompliance}
%Sharp bounds on $S_{\varepsilon\varepsilon,c}$ are 
%$
%\underline{S}_{\varepsilon\varepsilon,c}  
%\leq S_{\varepsilon\varepsilon,c} \leq 
%\overline{S}_{\varepsilon\varepsilon,c},
%$
%where
%\begin{eqnarray*}
%\underline{S}_{\varepsilon\varepsilon,c} =
%\int_0^1 \{   F_{1c}^{-1}(u)   - F_{0c}^{-1}(u)     \}^2 du,
%\quad 
%\overline{S}_{\varepsilon\varepsilon,c} =
%\int_0^1 \{   F_{1c}^{-1}(u)   - F_{0c}^{-1}(1-u)     \}^2 du
%\end{eqnarray*}
%are attainable when $\{ e_i'(1): U_i=c,i=1,\ldots,n  \}$ and $\{ e_i'(0) : U_i=c, i=1,\ldots, n  \}$ have the same ranks and opposite ranks, respectively. 
%If we further assume that $e_i'(1)$ and $e_i'(0)$ for those units with $U_i=c$ have nonnegative correlation, the upper bound can be further sharpened to $V_{1c} + V_{0c}$.
%\end{corollary}
%We defer discussion of estimating $F_{tc}(y)$ and other technical details to the Supplementary Material.

\subsubsection{Treatment effect decomposition}
Since there are two sources of variation---covariates and noncompliance---there are three possible $R^2$-type measures. First, we can measure the treatment effect variation explained by noncompliance alone (i.e., only $U$):
\begin{eqnarray*}
R^2_{\tau, U} =    \frac{S_{\tau\tau, U}}{  S_{\tau\tau}  }
=  \frac{S_{\tau\tau, U}}{  S_{\tau\tau, U} + \pi_c  S_{\tau\tau,c}  }
=  \frac{S_{\tau\tau, U}}{  S_{\tau\tau, U} + \pi_c S_{\delta\delta, c} + \pi_c S_{\varepsilon\varepsilon,c}  }.
\label{eq::r-u}
\end{eqnarray*}
Second, we can measure the proportion of treatment effect variation among Compliers explained by covariates (i.e., only $\bX$):
\begin{eqnarray*}
R^2_{\tau, c} =  \frac{  S_{\delta\delta, c}  } { S_{\tau\tau, c}    } 
= \frac{  S_{\delta\delta, c}  } { S_{\delta\delta, c} + S_{\varepsilon\varepsilon, c}    } .
\label{eq::r-c}
\end{eqnarray*}
Third, we can measure the treatment effect variation explained by covariates and noncompliance (i.e., both $\bX$ and $U$):
\begin{eqnarray*}
R^2_{\tau, U\bX} 
=\frac{   S_{\tau\tau,U}  + \pi_c S_{\delta \delta, c}  }{ S_{\tau\tau}    }
= \frac{   S_{\tau\tau,U}  + \pi_c S_{\delta \delta, c}  }{ S_{\tau\tau,U} + \pi_c S_{\delta\delta, c} + \pi_c S_{\varepsilon\varepsilon,c}    } .
\label{eq::r-ux}
\end{eqnarray*}

For each measure, we can use tailored versions of Corollary \ref{coro::independent} to construct bounds, or conduct sensitivity analysis as in Section \ref{sec::sensitivity-analysis}, with the sensitivity parameter expressed as the Spearman correlation between the treatment and control potential outcomes among Compliers.

\section{Simulation study}
\label{sec::examples}

\subsection{ITT estimators}

We simulate completely randomized experiments to evaluate the finite sample performance of the tests for systematic treatment effect variation based on $\widehat{\bbeta}_\OLS$, $\widehat{\bbeta}_\RI$, and $\widehat{\bbeta}_\RI^w$, the model-assisted version discussed in the Supplementary Material. Our data generation process is inspired by the Head Start Impact Study (HSIS) study analyzed in the next section.
For a given sample size, we first generate four independent covariates ($X_1$, a standard normal, $X_2$, a binary covariate with probability $0.5$ being $1$, $X_3$, a binary covariate with probability $0.25$ being $1$, and $X_4$, a standard normal). 
The control potential outcomes are then generated from
$$
Y_i(0) = 0.3 + 0.2 X_{1i} + 0.3  X_{2i} - 0.4 X_{3i} + 0.8 X_{4i} + u_i ,\qquad  u_i \sim \mathcal{N}(0, \sigma^2 ).
$$
We select $\sigma^2 = 0.26$ to make the marginal variance for the control potential outcomes 1; thus we can interpret impacts in ``effect size'' units.
The $R^2$ of regressing $Y(0)$ onto the covariates is approximately 0.74, due to the ``pre-test''-like variable $X_{4i}$.
Without $X_{4i}$, the $R^2$ is about 0.09.

The treatment effects are $\tau_i = \delta_i + \varepsilon_i$, with (i) either $\delta_i = 0.3$ for all $i$, or $\delta_i = 0.2 + 0.1 X_{1i} + 0.4 X_{3i}$; and (ii) either $\varepsilon_i = 0$ for all $i$, or $\varepsilon_i \sim \mathcal{N}(0,0.2^2)$.
All combinations of these two options give the four cases of (a) no treatment effect variation, (b) only systematic variation, (c) idiosyncratic variation with no systematic variation, and (d) both systematic and idiosyncratic variation.
For an $\alpha$-level test of systematic variation, scenarios (a) and (c) should only reject at rate $\alpha$, while we would like to see high rejection rates for scenarios (b) and (d). 
For scenario (d), the $R^2_\tau$ is about 0.5; systematic variation explains a good share of the overall variation.

To generate a synthetic dataset we generated all potential outcomes, randomized units into treatment with probability $0.6$, and then calculated the corresponding observed outcomes.
We then conducted a test for systematic variation using each of our three estimators.
For $\widehat{\bbeta}_\RI$ and $\widehat{\bbeta}_\OLS$ we use $X_1, X_2, X_3$.
For our covariate-adjusted estimator $\widehat{\bbeta}_\RI^w$ we also include the fairly predictive $X_4$ for adjustment.

Figure \ref{fg::power-itt} shows the power of these tests, with $\alpha=0.05$, for different sample sizes. First, all estimators appear asymptotically valid, consistent with the theoretical results. 
The OLS and adjusted estimators are slightly anti-conservative for small $n$, however, with rejection rates of around 9\%.
Second, the OLS estimator appears to have the greatest power in this setting, which is unsurprising since the true data generating process is a linear model. Finally, covariate adjustment slightly improves the power of the $\RI$ estimator. 
Overall, in the scenarios we consider, we only achieve decent levels of power in large samples, although there seems to be reasonable power for the sample size in the data application, $n = 3,586$.

\begin{figure}[t]
\includegraphics[width = \textwidth]{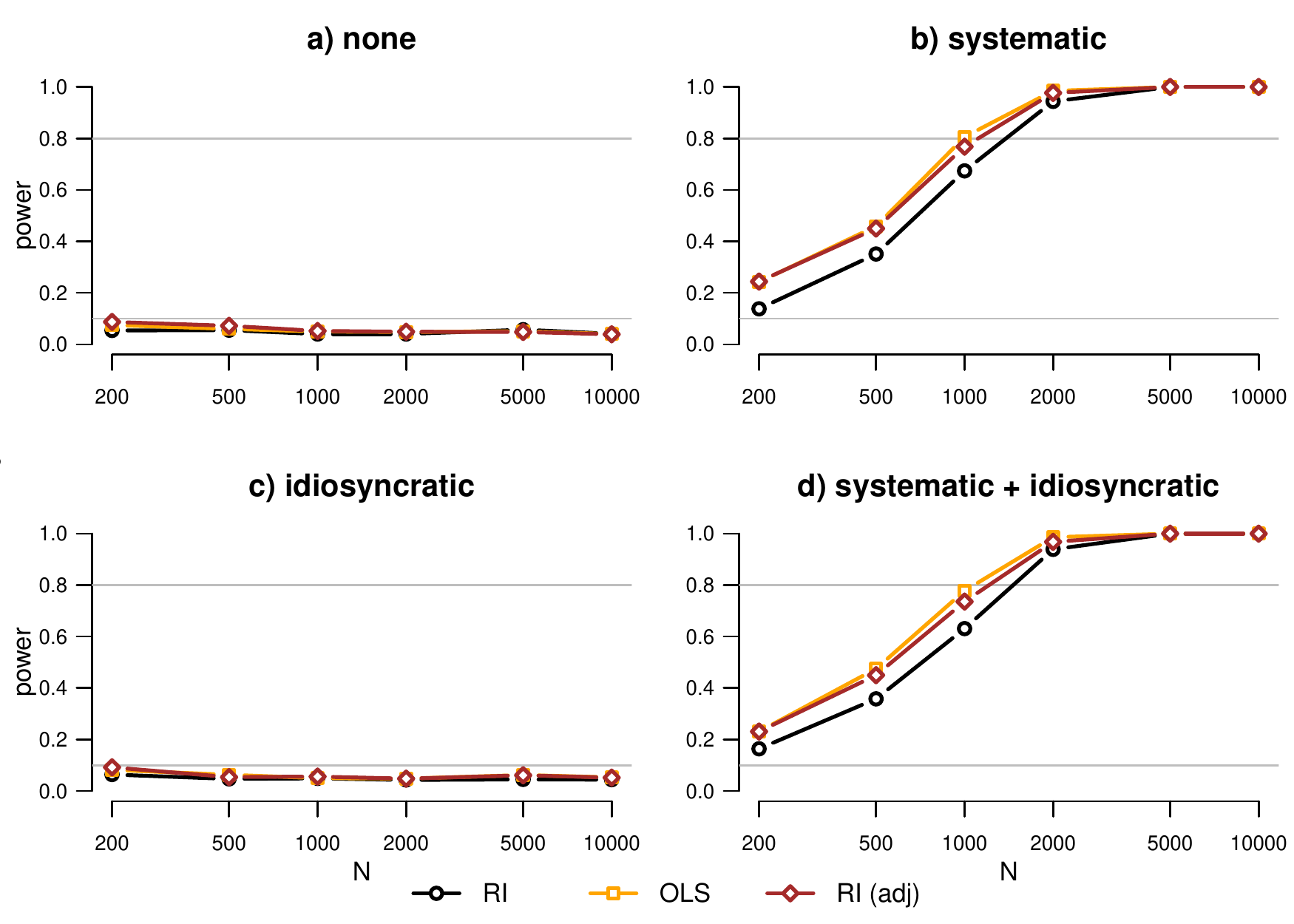}
\caption{Power of the tests based on $\widehat{\bbeta}_\RI$, $\widehat{\bbeta}_\OLS$, and $\widehat{\bbeta}_\RI^w$.}\label{fg::power-itt}
\end{figure}

\subsection{LATE estimators}

We next simulate completely randomized experiments with noncompliance to evaluate the finite sample performance of the tests for systematic treatment effect variation among Compliers based on $\widehat{\bbeta}_{c,\RI}$ and $\widehat{\bbeta}_\TSLS.$
We first generated a complete dataset as in the ITT case above, and then assigned strata membership to all units with probabilities proportional to their covariates.
For Always Takers we then set $Y_i(0) = Y_i(1)$, and for Never Takers, $Y_i(1) = Y_i(0)$.
The overall ITT is now reduced to 0.21 (due to the 0 effects of Never Takers and Always Takers), although the CACE is still approximately 0.3.
The proportion of Compliers is approximately 68\%.

The Compliers have the systematic and idiosyncratic effects described as above.
We tested for the presence of systematic variation for Compliers under the exclusion restrictions.
Figure \ref{fg::power-late} shows the power of these tests for our $\RI$ and $\TSLS$ estimators.
First, in this scenario, the 2SLS and the RI estimators are virtually equivalent; the additional adjustment provided by $\TSLS$ does not add significantly to the precision.
We see the tests are valid (they even appear conservative) for cases (a) and (c).  
Power is reduced compared to the ITT simulation; this is reasonable as power is effectively a function of the number of Compliers, with additional uncertainty due to partial information about the identity of Compliers.

\begin{figure}[ht]
\includegraphics[width = \textwidth]{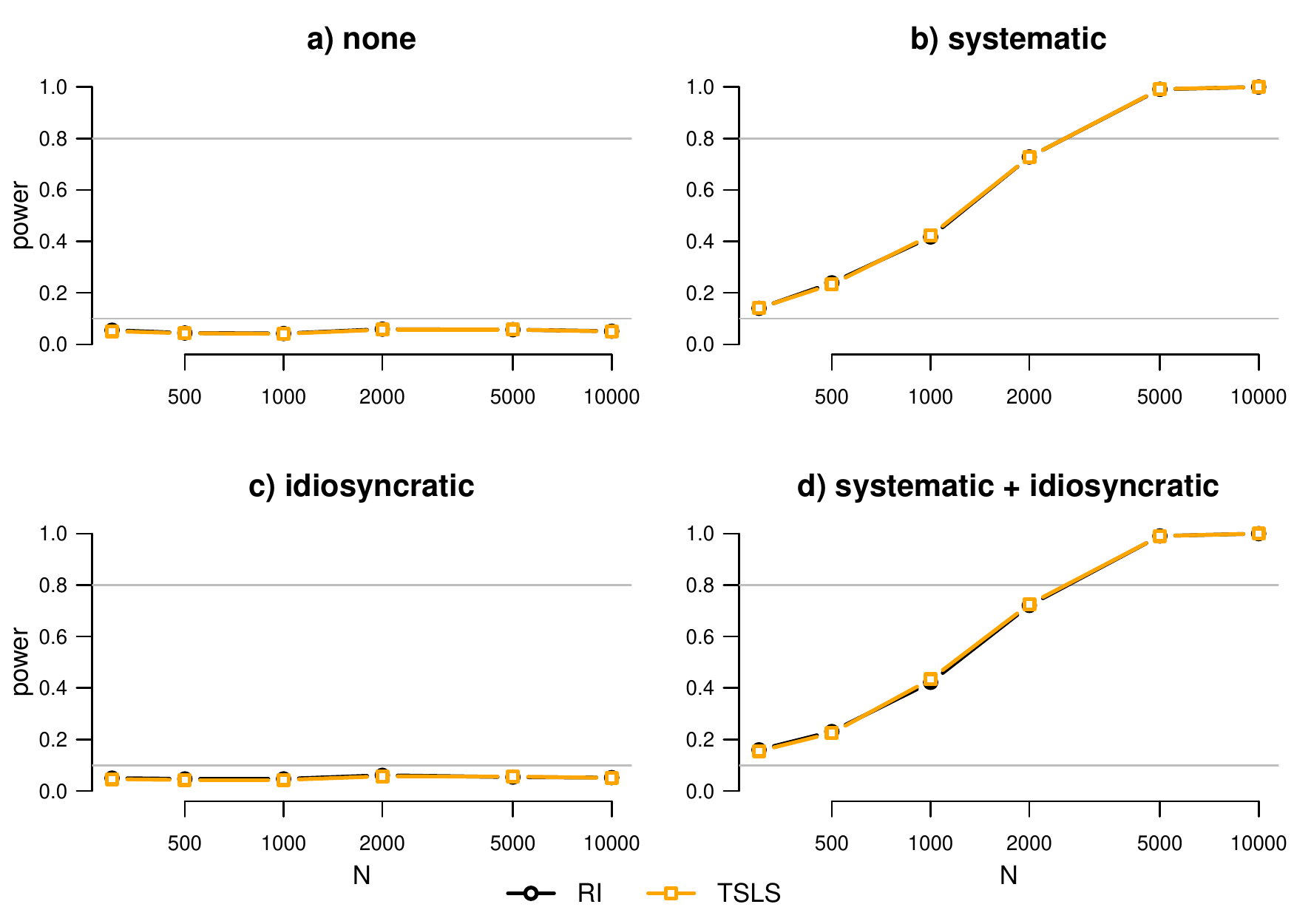}
\caption{Power of the tests based on $\widehat{\bbeta}_{c,\RI}$ and $\widehat{\bbeta}_\TSLS.$}\label{fg::power-late}
\end{figure}

\section{Application to the Head Start Impact Study}
\label{sec::hsis}

Established in 1965, Head Start is the largest Federal preschool program in the United States, serving nearly 1 million low-income three- and four-year-old children each year at a cost of over \$7 billion~\citep{acf2015}. Researchers and policymakers have debated Head Start's effectiveness since its inception, with early randomized trials finding limited impacts~\citep[e.g.,][]{westinghouse1969} and quasi-experimental studies showing much larger effects~\citep[e.g.,][]{Currie:1993wo}. Designed in part to settle this debate, the Head Start Impact Study (HSIS) is a large-scale, nationally representative randomized trial of Head Start first launched in 2002~\citep{puma2010head}.  
The Congressional mandate for HSIS included two broad questions: (1) the program's {\em overall} impact, and (2) how impacts {\em vary} across children and centers. The policy debate has largely focused on this first question; 
HSIS only found modest average effects on a range of children's cognitive and social-emotional outcomes. However, both the original study and several recent papers argue that these topline results mask important treatment effect variation~\citep[e.g.,][]{bloom_weiland2014, Bitler:2014ve, ding::2015jrssb, walters2015inputs, feller2016hsis}. Understanding such variation is critical both for assessing the program's benefits and costs and for improving the practice and science of early childhood education. 

HSIS collected a rich set of covariates about children and their families, including pre-test score, child's age, child's race, child's home language, mother's education level, and mother's marital status. At the same time, many potentially important covariates are unavailable. For instance, while families must be low-income to be eligible for Head Start, HSIS does not include information on families' actual income nor other financial details that could be important predictors of program impact. In addition,~\citet{feller2016hsis} and others argue that that the setting in which a child would otherwise receive care is an important source of impact variation, although this is not directly observable. 

We now use the methods outlined above to assess treatment effect variation in HSIS. The original study included $n = 4,400$ total children, with $n_1 = 2,644$ in the treatment group and $n_0 = 1,796$ in the control group. Following earlier analyses~\citep{ding::2015jrssb} and to simplify exposition, we restrict our attention to a complete-case subset of the HSIS, with $n_1 = 2,238$ in the treatment group and $n_0 = 1,348$ in the control group (so $p_1 \approx 0.62$ and $p_0\approx 0.38$). 
Our outcome of interest is the Peabody Picture Vocabulary Test (PPVT), a widely used measure of cognitive ability in early childhood. To assess treatment effect variation, we consider the full set of child- and family-level covariates used in the original HSIS analysis of~\citet{puma2010head}, including those mentioned above. After creating dummy variables for factors (e.g., re-coding race), the covariate matrix has $17$ columns. See Figure~\ref{fg::indiv_var_R2} for a complete list.

\subsection{Decomposing variation in the ITT effect}
We first explore treatment effect variation for the ITT estimate, beginning with estimating systematic treatment effect variation. We examine three estimators: the randomization-based and OLS estimators discussed in Section~\ref{sec::systematic}, $\widehat{\bbeta}_\RI$ and $\widehat{\bbeta}_\OLS$, and the corresponding model-assisted version of the RI estimator discussed in the Supplementary Material, $\widehat{\bbeta}_\RI^w$. For this latter estimator, we use all available covariates to adjust the standard estimators, that is, $\bW$ is the entire vector of covariates.

\paragraph{Omnibus test for systematic treatment effect variation.} 
We begin by using these estimators for an omnibus test of whether any treatment effect variation is explained by the full set of covariates. 
The $p$-values for the unadjusted $\widehat{\bbeta}_\RI$ estimator and model-assisted $\widehat{\bbeta}_\RI^w$ are $0.39$ and $0.25$, respectively, which do not show any evidence of treatment effect variation. 
The OLS estimator, however, shows much stronger evidence with $p = 0.005$. 

Importantly, all three estimators are based on the same underlying assumptions: the randomization itself justifies all three $p$-values. And while we expect the unadjusted $\widehat{\bbeta}_\RI$ to have the lowest power, it is instructive that the $p$-value for $\widehat{\bbeta}_\OLS$ is substantially smaller than the $p$-value for the covariate-adjusted $\widehat{\bbeta}_\RI^w$. 
%
%Thus, the OLS estimator is more powerful than the other estimators in this setting because it better exploits the available covariate information. 
%
As we discuss in Section~\ref{sec:ols_int}, $\widehat{\bbeta}_\OLS$ can account for covariate imbalance across experimental arms by estimating the $\bS_{xx}$ matrix separately for the treatment and control groups. By contrast, $\widehat{\bbeta}_\RI$ does not address imbalance in $\bX$ and instead attempts to residualize out the $Y$ in order to get a more precise estimate of the relationship of the $\bX$ to $Y$ for each treatment arm. Based on the discrepancy in $p$-values, adjusting for baseline imbalance is clearly important in this example. 
%In this context, apparently, the imbalance was more important to account for, hence the discrepancy in the $p$-values between $\widehat{\bbeta}_\OLS$ and $\widehat{\bbeta}_\RI^w$ (we would expect the lower-powered, completely unadjusted $\widehat{\bbeta}_\RI$ to have a higher $p$-value in general).

%The lower OLS $p$-value is \emph{not} due to additional assumptions imposed.
%All the estimators are based on the same assumptions and that no specific modeling form is needed for validity.  
%The differences in estimators is purely from different powers due to different ways of incorporating covariate information (as we saw in the simulation studies).  
%Future work would attempt to combine these two forms of adjustment to simultaneously account for covariate imbalance as well as estimating the  relationship between $\bX$ and $Y$ more precisely.

\paragraph{Treatment effect $R^2_\tau$.} Next, we examine how much of the variation could be explained by our covariates.
Figure~\ref{fg::head_start_R2} shows values of the treatment effect $R^2_\tau$ using $\widehat{\bbeta}^w_\RI$ to estimate the systematic variation. Results are nearly identical using the other estimators. In the worst case of perfect negative dependence between potential outcomes (not shown), the treatment effect $R^2_\tau$ could be as low as 0.01. Assuming that this dependence is nonnegative, the treatment effect $R^2_\tau$ ranges from 0.03 to 0.76. While the estimate is clearly sensitive to the unidentifiable sensitivity parameter, the covariates explain a substantial proportion of treatment effect variation for values of $\rho$ near 1.

\begin{figure}

	\centering
	\begin{subfigure}[b]{0.48\textwidth}
		\includegraphics[width = \textwidth]{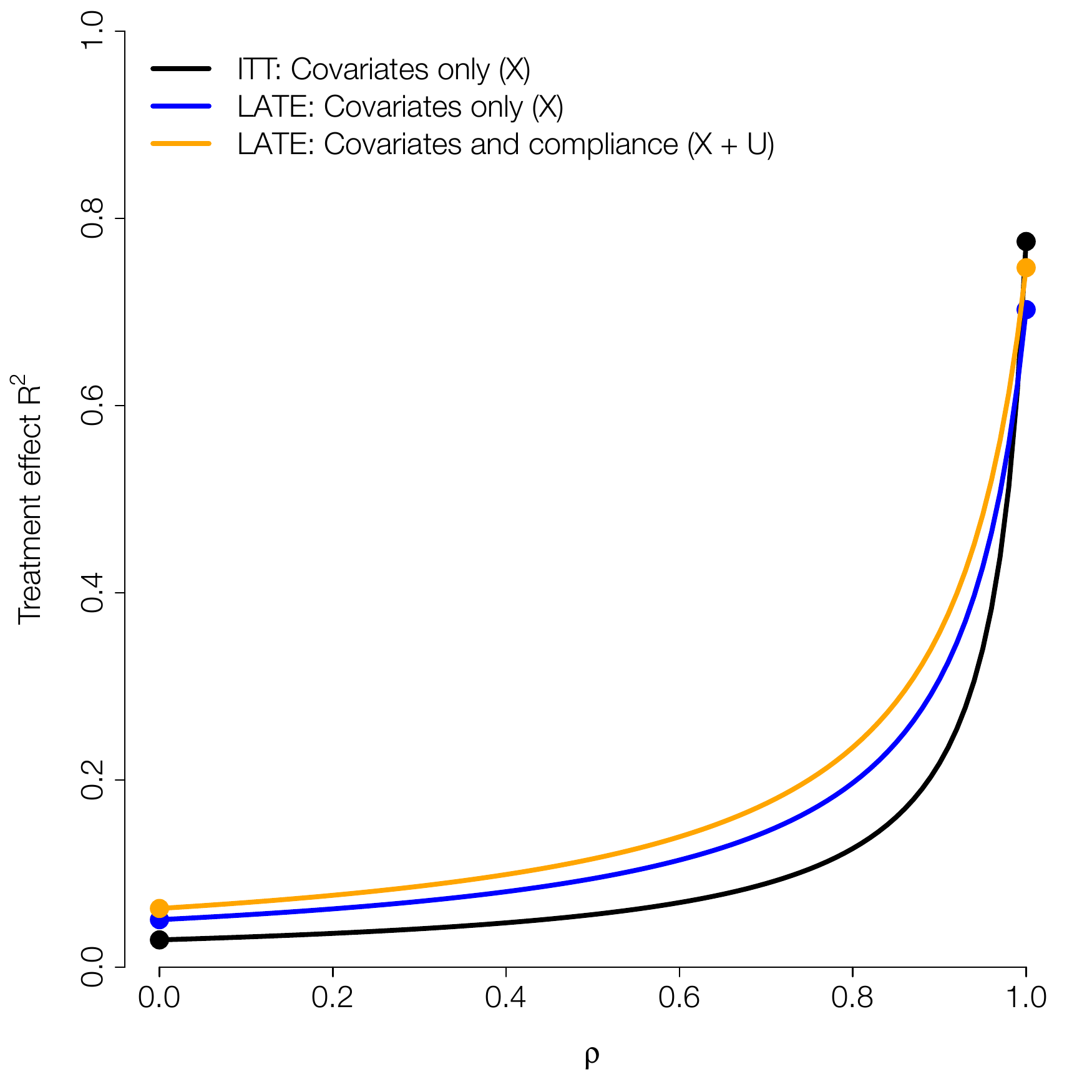}
		\caption{Overall $R^2_\tau$}\label{fg::head_start_R2}
	\end{subfigure}%
	\quad\begin{subfigure}[b]{0.48\textwidth}
		\includegraphics[width = \textwidth]{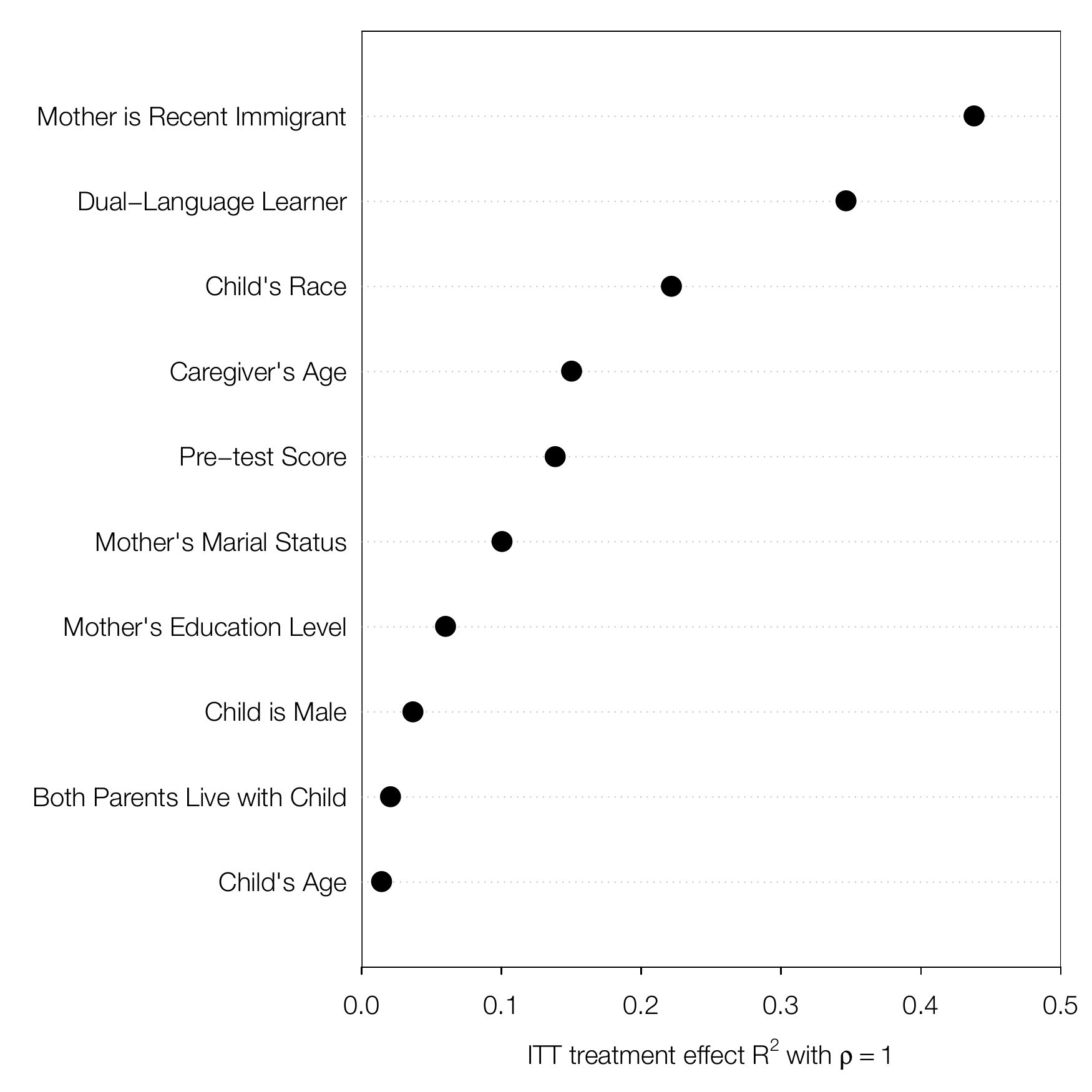}
		\caption{$R^2_\tau$ separately by covariate}\label{fg::indiv_var_R2}
	\end{subfigure}
	
	\caption{Treatment effect $R^2_\tau$, with sensitivity parameter, $\rho \in [0,1]$.}

\end{figure}

We can also use this framework to assess the relative importance of each covariate in terms of explaining overall treatment effect variation. To do this, we use the model-assisted RI estimator, $\widehat{\bbeta}_\RI^w$, adjusting for all covariates (i.e., $\text{dim}(\bW) = 17$) but restricting systematic treatment effect variation to one covariate at a time. Note that we consider factors (e.g., race) as a group. Figure~\ref{fg::indiv_var_R2} shows the resulting estimates for the upper bound of $R^2_\tau$, with lower bound estimates all below 0.01.  Having a mother who is a recent immigrant and dual language learner status (which are highly correlated in practice) could each explain a substantial proportion of treatment effect variation, consistent with previous results from~\citet{bloom_weiland2014} and~\citet{Bitler:2014ve}. This is not true for other covariates, like mother's education level.

\paragraph{Negative correlation between treatment effect and control potential outcomes.} Finally, we test whether the individual-level idiosyncratic treatment effects, $\{\varepsilon_i\}_{i=1}^n$, are negatively correlated with the control potential outcomes, $\{Y_i(0)\}_{i=1}^n$, extending results from~\citet{raudenbush::2015}. As outlined in the Supplementary Material, we do so by testing whether the variance of $\{  Y_i^\obs  - \bX_i^\T  \widehat{\bbeta}^w_\RI: T_i = 1 \}$ is smaller than the variance of $  \{  Y_i^\obs: T_i = 0 \}$. This yields a $p$-value of $ 0.02$, which suggests that the unexplained treatment effect is indeed larger for smaller values of the control potential outcomes. This result is consistent with findings from~\citet{Bitler:2014ve} who use a quantile treatment effect approach.

\subsection{Incorporating noncompliance}
As with many social experiments, there is substantial noncompliance with random assignment in HSIS. In the analysis sample we consider here, the estimated proportion of compliance types is  $\widehat{\pi}_c = 0.69$ for Compliers,  $\widehat{\pi}_a = 0.13$ for Always Takers, and $\widehat{\pi}_n = 0.18$ for Never Takers. Given the exclusion restrictions for Always Takers and Never Takers, the treatment effect is therefore zero (by assumption) for over 30 percent of the sample, suggesting that noncompliance will be an important component of treatment effect variation. 

In the setting with noncompliance, we focus on two estimators for systematic treatment effect variation among Compliers: the randomization-based estimator, $\widehat{\bbeta}_{c,\RI}$, and the Two-Stage Least Squares estimator, $\widehat{\bbeta}_\TSLS$. We first use these estimators to construct omnibus tests for systematic treatment effect variation among Compliers. Tests using both estimators show strong evidence for such variation, with $p$-value $ 0.02$ using $\widehat{\bbeta}_{c,\RI}$ and $p$-value $ 0.01$ using $\widehat{\bbeta}_\TSLS$.

Finally, we turn to decomposing the overall treatment effect. As in the ITT case, we assume that the potential outcomes have a nonnegative correlation. Figure~\ref{fg::head_start_R2} shows the treatment effect $R^2$ among Compliers, which ranges from $R^2_{\tau,c} = 0.05$ to $R^2_{\tau,c} = 0.68$. 
Next, we can calculate treatment effect variation due to noncompliance, $R^2_{\tau,U}$. In the case of HSIS, this is relatively small---between 0.01 and 0.16---in part because the overall treatment effect is fairly small. Therefore, the overall treatment effect decomposition due to both covariates and noncompliance, $R^2_{\tau, U\bX}$, is quite close to $R^2_{\tau, c}$, as shown in Figure~\ref{fg::head_start_R2}. Taken together, these estimates suggest that there is indeed important treatment effect variation that is neither captured by pre-treatment covariates nor by noncompliance, consistent with previous results in~\citet{ding::2015jrssb}.

\section{Conclusion}
\label{sec::discussion}
In this paper, we propose a broad, flexible framework for assessing and decomposing treatment effect variation in randomized experiments with and without noncompliance. In general, we believe this is a natural setup for researchers to formulate and investigate a broad range of questions about impact heterogeneity~\citep[e.g.,][]{heckman::1997}. Applications include assessing underlying causal mechanisms and targeting treatments based on individual-level characteristics. Understanding such variation is also important for the design of experiments.~\citet{djebbari::2008}, for example, argue that characterizing the size of the idiosyncratic treatment effect is useful for determining the value of additional data collection.

We briefly note several directions for future work. 
First, our primary purpose was to propose a framework for analysis rooted in and justified by the randomization itself.
As a result, we focused on the core properties of several relatively simple versions of linear regression and TSLS. 
We did not, however, fully explore their practical and finite-sample properties.
For example, in future work, we hope to determine the settings in which model assistance will most improve estimation and assess the increased power of the $\OLS$ approach versus the unbiased $\RI$ approach.
We are also investigating how to connect model assisted and $\OLS$ approaches to take advantage of both methods of precision gain.
Similarly, there is still much potential improvement in determining ways of characterizing the degree of heterogeneity, such as with an effect size for the systematic variation.

Second, a natural extension is to use more complex methods to estimate systematic treatment effects, such as via hierarchical models~\citep{feller2015hierarchical} or via machine learning methods~\citep{wager2015estimation}, extending the results for the omnibus test and treatment effect $R^2_\tau$ accordingly. While the guarantees from randomization are clearly weaker in such settings, researchers can assess these tradeoffs themselves. For example, hierarchical modeling would be especially useful in the Head Start Impact Study due to the multi-site design~\citep{bloom_weiland2014}. 

Third, a question of increasing practical importance is the generalizability of experimental results to a given target population~\citep{stuart2011use}. We believe that the treatment effect $R^2_\tau$ is a critical measure for assessing the credibility of these generalizations. In short, if there is substantial idiosyncratic treatment effect variation, i.e., $R^2_\tau$ is small, then researchers should be wary of using observed covariates to extrapolate treatment effects.

Finally, a question is how to extend this treatment effect variation framework to non-randomized settings. While the results would necessarily rest on much stronger assumptions, many settings already use an as-if-randomized framework, such as in observational studies~\citep{rosenbaum::2002, imbens::2015}. 
Under this approach, extensions should be natural.

%\clearpage
\bibliographystyle{abbrvnat}
\bibliography{referencesTEV}

\begin{thebibliography}{59}
\providecommand{\natexlab}[1]{#1}
\providecommand{\url}[1]{\texttt{#1}}
\expandafter\ifx\csname urlstyle\endcsname\relax
  \providecommand{\doi}[1]{doi: #1}\else
  \providecommand{\doi}{doi: \begingroup \urlstyle{rm}\Url}\fi

\bibitem[Abadie(2003)]{abadie::2003}
A.~Abadie.
\newblock Semiparametric instrumental variable estimation of treatment response
  models.
\newblock \emph{Journal of Econometrics}, 113:\penalty0 231--263, 2003.

\bibitem[{Administration for Children and Families}(2015)]{acf2015}
{Administration for Children and Families}.
\newblock {Head Start} program facts, fiscal year 2014.
\newblock Available at
  https://eclkc.ohs.acf.hhs.gov/hslc/data/factsheets/docs/hs-program-fact-sheet-2014.pdf,
  2015.

\bibitem[Angrist and Pischke(2008)]{angrist::2008}
J.~D. Angrist and J.~Pischke.
\newblock \emph{Mostly Harmless Econometrics: An Empiricist's Companion}.
\newblock Princeton: Princeton University Press, 2008.

\bibitem[Angrist et~al.(1996)Angrist, Imbens, and Rubin]{angrist::1996}
J.~D. Angrist, G.~W. Imbens, and D.~B. Rubin.
\newblock Identification of causal effects using instrumental variables.
\newblock \emph{Journal of the American Statistical Association}, 91:\penalty0
  444--455, 1996.

\bibitem[Angrist et~al.(2013)Angrist, Pathak, and
  Walters]{angrist2013explaining}
J.~D. Angrist, P.~A. Pathak, and C.~R. Walters.
\newblock Explaining charter school effectiveness.
\newblock \emph{American Economic Journal: Applied Economics}, 5\penalty0
  (4):\penalty0 1--27, 2013.

\bibitem[Aronow et~al.(2014)Aronow, Green, and Lee]{aronow::2014}
P.~M. Aronow, D.~P. Green, and D.~K. Lee.
\newblock Sharp bounds on the variance in randomized experiments.
\newblock \emph{The Annals of Statistics}, 42:\penalty0 850--871, 2014.

\bibitem[Athey and Imbens(2016)]{Athey:2015vu}
S.~Athey and G.~Imbens.
\newblock Recursive partitioning for heterogeneous causal effects.
\newblock \emph{Proceedings of the National Academy of Sciences}, 113\penalty0
  (27):\penalty0 7353--7360, 2016.

\bibitem[Berrington~de Gonz{\'a}lez and Cox(2007)]{berrington::2007}
A.~Berrington~de Gonz{\'a}lez and D.~R. Cox.
\newblock {Interpretation of interaction: A review}.
\newblock \emph{The Annals of Applied Statistics}, 1:\penalty0 371--385, 2007.

\bibitem[Bitler et~al.(2014)Bitler, Hoynes, and Domina]{Bitler:2014ve}
M.~Bitler, H.~Hoynes, and T.~Domina.
\newblock {Experimental Evidence on Distributional Effects of Head Start}.
\newblock Working Paper, 2014.

\bibitem[Blinder(1973)]{blinder1973wage}
A.~S. Blinder.
\newblock Wage discrimination: reduced form and structural estimates.
\newblock \emph{Journal of Human resources}, 8:\penalty0 436--455, 1973.

\bibitem[Bloom and Weiland(2014)]{bloom_weiland2014}
H.~S. Bloom and C.~Weiland.
\newblock To what extent do the effects of {Head Start} on enrolled children
  vary across sites?
\newblock Working Paper, 2014.

\bibitem[Bound et~al.(1995)Bound, Jaeger, and Baker]{bound1995problems}
J.~Bound, D.~A. Jaeger, and R.~M. Baker.
\newblock Problems with instrumental variables estimation when the correlation
  between the instruments and the endogenous explanatory variable is weak.
\newblock \emph{Journal of the American statistical association}, 90:\penalty0
  443--450, 1995.

\bibitem[Cochran(1977)]{cochran::1977}
W.~G. Cochran.
\newblock \emph{Sampling Techniques}.
\newblock New York: John Wiley \& Sons, 3rd edition, 1977.

\bibitem[Cox(1984)]{cox1984interaction}
D.~R. Cox.
\newblock Interaction (with discussion).
\newblock \emph{International Statistical Review}, 52:\penalty0 1--24, 1984.

\bibitem[Crump et~al.(2008)Crump, Hotz, Imbens, and Mitnik]{crump::2008}
R.~K. Crump, V.~J. Hotz, G.~W. Imbens, and O.~A. Mitnik.
\newblock {Nonparametric tests for treatment effect heterogeneity}.
\newblock \emph{Review of Economics and Statistics}, 90:\penalty0 389--405,
  2008.

\bibitem[Currie and Thomas(1995)]{Currie:1993wo}
J.~Currie and D.~Thomas.
\newblock {Does Head Start make a difference?}
\newblock \emph{American Economic Review}, 85\penalty0 (3):\penalty0 341--364,
  1995.

\bibitem[Ding(2014)]{ding::2014}
P.~Ding.
\newblock A paradox from randomization-based causal inference.
\newblock \emph{arXiv preprint arXiv:1402.0142}, 2014.

\bibitem[Ding et~al.(2016)Ding, Feller, and Miratrix]{ding::2015jrssb}
P.~Ding, A.~Feller, and L.~W. Miratrix.
\newblock Randomization inference for treatment effect variation.
\newblock \emph{Journal of the Royal Statistical Society, Series B (Statistical
  Methodology)}, 78:\penalty0 655--671, 2016.

\bibitem[Djebbari and Smith(2008)]{djebbari::2008}
H.~Djebbari and J.~Smith.
\newblock {Heterogeneous impacts in PROGRESA}.
\newblock \emph{Journal of Econometrics}, 145:\penalty0 64--80, 2008.

\bibitem[Feller and Gelman(2015)]{feller2015hierarchical}
A.~Feller and A.~Gelman.
\newblock Hierarchical models for causal effects.
\newblock \emph{Emerging Trends in the Social and Behavioral Sciences: An
  Interdisciplinary, Searchable, and Linkable Resource}, 2015.

\bibitem[Feller et~al.(2016)Feller, Grindal, Miratrix, and
  Page]{feller2016hsis}
A.~Feller, T.~Grindal, L.~Miratrix, and L.~C. Page.
\newblock Compared to what? variation in the impacts of early childhood
  education by alternative care type.
\newblock \emph{The Annals of Applied Statistics}, 10\penalty0 (3):\penalty0
  1245--1285, 2016.

\bibitem[Fisher(1935)]{fisher::1935a}
R.~A. Fisher.
\newblock \emph{The Design of Experiments.}
\newblock Edinburgh: Oliver \& Boyd, 1st edition, 1935.

\bibitem[Fogarty(2016)]{fogarty2016regression}
C.~B. Fogarty.
\newblock Regression assisted inference for the average treatment effect in
  paired experiments.
\newblock \emph{arXiv preprint arXiv:1612.05179}, 2016.

\bibitem[Frangakis and Rubin(2002)]{frangakis::2002}
C.~E. Frangakis and D.~B. Rubin.
\newblock Principal stratification in causal inference.
\newblock \emph{Biometrics}, 58:\penalty0 21--29, 2002.

\bibitem[Fr{\'e}chet(1951)]{frechet::1951}
M.~Fr{\'e}chet.
\newblock Sur les tableaux de corr\'elation dont les marges son donn\'ees.
\newblock \emph{Annals Universite de Lyon, Sect. A. Ser. 3}, 14:\penalty0
  53--77, 1951.

\bibitem[Green and Kern(2012)]{Green:2012kt}
D.~P. Green and H.~L. Kern.
\newblock {Modeling Heterogeneous Treatment Effects in Survey Experiments with
  Bayesian Additive Regression Trees}.
\newblock \emph{The Public Opinion Quarterly}, 76:\penalty0 491--511, 2012.

\bibitem[H{\'a}jek(1960)]{hajek::1960}
J.~H{\'a}jek.
\newblock Limiting distributions in simple random sampling from a finite
  population.
\newblock \emph{Publications of the Mathematics Institute of the Hungarian
  Academy of Science}, 5:\penalty0 361--74, 1960.

\bibitem[Heckman et~al.(1997)Heckman, Smith, and Clements]{heckman::1997}
J.~J. Heckman, J.~Smith, and N.~Clements.
\newblock {Making the most out of programme evaluations and social experiments:
  Accounting for heterogeneity in programme impacts}.
\newblock \emph{The Review of Economic Studies}, 64:\penalty0 487--535, 1997.

\bibitem[Hill(2011)]{hill2011bayesian}
J.~L. Hill.
\newblock Bayesian nonparametric modeling for causal inference.
\newblock \emph{Journal of Computational and Graphical Statistics},
  20:\penalty0 217--240, 2011.

\bibitem[Hoeffding(1941)]{hoeffding::1941}
W.~Hoeffding.
\newblock Masstabinvariante korrelationsmasse f\"ur diskontinuierliche
  verteilungen.
\newblock \emph{Arkiv fr matematischen Wirtschaften und Sozialforschung},
  7:\penalty0 49--70, 1941.

\bibitem[Huang et~al.(2012)Huang, Gilbert, and Janes]{huang2012assessing}
Y.~Huang, P.~B. Gilbert, and H.~Janes.
\newblock Assessing treatment-selection markers using a potential outcomes
  framework.
\newblock \emph{Biometrics}, 68:\penalty0 687--696, 2012.

\bibitem[Huber(1967)]{huber::1967}
P.~J. Huber.
\newblock The behavior of maximum likelihood estimates under nonstandard
  conditions.
\newblock In \emph{Proceedings of the Fifth Berkeley Symposium on Mathematical
  Statistics and Probability}, volume~1, pages 221--233, 1967.

\bibitem[Imai and Ratkovic(2013)]{imai2013estimating}
K.~Imai and M.~Ratkovic.
\newblock Estimating treatment effect heterogeneity in randomized program
  evaluation.
\newblock \emph{The Annals of Applied Statistics}, 7:\penalty0 443--470, 2013.

\bibitem[Imbens(2014)]{imbens2014instrumental}
G.~Imbens.
\newblock Instrumental variables: An econometrician's perspective (with
  discussion).
\newblock \emph{Statistical Science}, 29:\penalty0 323--358, 2014.

\bibitem[Imbens and Rubin(2015)]{imbens::2015}
G.~W. Imbens and D.~B. Rubin.
\newblock \emph{Causal Inference in Statistics, and in the Social and
  Biomedical Sciences}.
\newblock New York: Cambridge University Press, 2015.

\bibitem[Kempthorne(1952)]{kempthorne::1952}
O.~Kempthorne.
\newblock \emph{The Design and Analysis of Experiments.}
\newblock New York: Wiley, 1952.

\bibitem[Lehmann(1998)]{lehmann::1998}
E.~L. Lehmann.
\newblock \emph{Elements of Large-Sample Theory}.
\newblock New York: Springer, 1998.

\bibitem[Li and Ding(2017)]{li2017general}
X.~Li and P.~Ding.
\newblock General forms of finite population central limit theorems with
  applications to causal inference.
\newblock \emph{Journal of the American Statistical Association}, page in
  press, 2017.

\bibitem[Lin(2013)]{lin::2013}
W.~Lin.
\newblock Agnostic notes on regression adjustments to experimental data:
  reexamining {F}reedman's critique.
\newblock \emph{The Annals of Applied Statistics}, 7:\penalty0 295--318, 2013.

\bibitem[Matsouaka et~al.(2014)Matsouaka, Li, and Cai]{matsouaka2014evaluating}
R.~A. Matsouaka, J.~Li, and T.~Cai.
\newblock Evaluating marker-guided treatment selection strategies.
\newblock \emph{Biometrics}, 70:\penalty0 489--499, 2014.

\bibitem[Middleton and Aronow(2015)]{middleton2016unbiased}
J.~A. Middleton and P.~M. Aronow.
\newblock Unbiased estimation of the average treatment effect in
  cluster-randomized experiments.
\newblock \emph{Statistics, Politics and Policy}, 6:\penalty0 39--75, 2015.

\bibitem[Nelsen(2007)]{nelsen::2007}
R.~B. Nelsen.
\newblock \emph{An Introduction to Copulas}.
\newblock New York: Springer, 2nd edition, 2007.

\bibitem[Neyman(1923)]{neyman::1923}
J.~Neyman.
\newblock On the application of probability theory to agricultural experiments.
  {E}ssay on principles. {S}ection 9.
\newblock \emph{Statistical Science}, 5:\penalty0 465--472, 1923.

\bibitem[Oaxaca(1973)]{oaxaca1973male}
R.~Oaxaca.
\newblock Male-female wage differentials in urban labor markets.
\newblock \emph{International Economic Review}, 14:\penalty0 693--709, 1973.

\bibitem[Puma et~al.(2010)Puma, Bell, Cook, Heid, Shapiro, Broene, Jenkins,
  Fletcher, Quinn, Friedman, et~al.]{puma2010head}
M.~Puma, S.~Bell, R.~Cook, C.~Heid, G.~Shapiro, P.~Broene, F.~Jenkins,
  P.~Fletcher, L.~Quinn, J.~Friedman, et~al.
\newblock Head start impact study: Final report.
\newblock Technical report, Department of Health and Human Services,
  Administration for Children and Families, Washington DC, 2010.

\bibitem[Raudenbush and Bloom(2015)]{raudenbush::2015}
S.~W. Raudenbush and H.~S. Bloom.
\newblock Learning about and from a distribution of program impacts using
  multisite trials.
\newblock \emph{American Journal of Evaluation}, 36\penalty0 (4):\penalty0
  475--499, 2015.

\bibitem[Rosenbaum(1999)]{rosenbaum1999reduced}
P.~R. Rosenbaum.
\newblock Reduced sensitivity to hidden bias at upper quantiles in
  observational studies with dilated treatment effects.
\newblock \emph{Biometrics}, 55:\penalty0 560--564, 1999.

\bibitem[Rosenbaum(2002)]{rosenbaum::2002}
P.~R. Rosenbaum.
\newblock \emph{Observational Studies}.
\newblock New York: Springer, 2nd edition, 2002.

\bibitem[Rosenbaum(2007)]{rosenbaum2007confidence}
P.~R. Rosenbaum.
\newblock Confidence intervals for uncommon but dramatic responses to
  treatment.
\newblock \emph{Biometrics}, 63:\penalty0 1164--1171, 2007.

\bibitem[Rubin(1974)]{rubin::1974}
D.~B. Rubin.
\newblock Estimating causal effects of treatments in randomized and
  nonrandomized studies.
\newblock \emph{Journal of Educational Psychology}, 66:\penalty0 688--701,
  1974.

\bibitem[Rubin(1980)]{rubin::1980}
D.~B. Rubin.
\newblock Comment on ``{R}andomization analysis of experimental data: the
  {F}isher randomization test'' by {D}. {B}asu.
\newblock \emph{Journal of the American Statistical Association}, 75:\penalty0
  591--593, 1980.

\bibitem[S{\"a}rndal et~al.(2003)S{\"a}rndal, Swensson, and
  Wretman]{sarndal2003model}
C.-E. S{\"a}rndal, B.~Swensson, and J.~Wretman.
\newblock \emph{Model-Assisted Survey Sampling}.
\newblock New York: Springer, 2003.

\bibitem[Staiger and Stock(1997)]{staiger1997instrumental}
D.~O. Staiger and J.~H. Stock.
\newblock Instrumental variables regression with weak instruments.
\newblock \emph{Econometrica}, 65:\penalty0 557--586, 1997.

\bibitem[Stuart et~al.(2011)Stuart, Cole, Bradshaw, and Leaf]{stuart2011use}
E.~A. Stuart, S.~R. Cole, C.~P. Bradshaw, and P.~J. Leaf.
\newblock The use of propensity scores to assess the generalizability of
  results from randomized trials.
\newblock \emph{Journal of the Royal Statistical Society: Series A (Statistics
  in Society)}, 174\penalty0 (2):\penalty0 369--386, 2011.

\bibitem[Wager and Athey(2017)]{wager2015estimation}
S.~Wager and S.~Athey.
\newblock Estimation and inference of heterogeneous treatment effects using
  random forests.
\newblock \emph{Journal of the American Statistical Association}, \penalty0
  (just-accepted), 2017.

\bibitem[Walters(2015)]{walters2015inputs}
C.~R. Walters.
\newblock Inputs in the production of early childhood human capital: Evidence
  from head start.
\newblock \emph{American Economic Journal: Applied Economics}, 7\penalty0
  (4):\penalty0 76--102, 2015.

\bibitem[{Westinghouse Learning Corporation}(1969)]{westinghouse1969}
{Westinghouse Learning Corporation}.
\newblock \emph{{The Impact of Head Start: An Evaluation of the Effects of Head
  Start on Children's Cognitive and Affective Development, Volume 1: Report to
  the Office of Economic Opportunity}}.
\newblock Athens, Ohio: Westinghouse Learning Corporation and Ohio University,
  1969.

\bibitem[White(1980)]{white::1980}
H.~White.
\newblock A heteroskedasticity-consistent covariance matrix estimator and a
  direct test for heteroskedasticity.
\newblock \emph{Econometrica}, pages 817--838, 1980.

\bibitem[Zhao et~al.(2017)Zhao, Ding, Mukerjee, and
  Dasgupta]{zhao2016randomization}
A.~Zhao, P.~Ding, R.~Mukerjee, and T.~Dasgupta.
\newblock Randomization-based causal inference from split-plot designs.
\newblock \emph{Annals of Statistics}, page in press, 2017.

\end{thebibliography}

%%%%%%%%%%%%%%%%%%%%%%%%%%%%%%%%%%%%%%%%%%%%%%%%%%%%%%%%%%%%%%%%%%%%%%%%%%%%%%%%%%%%%%%%%%%%%%%%%%%%
%%%%%%%%%%%%%%%%%%%%%%%%%%%%%%%%%%%%%%%%%%%%%%%%%%%%%%%%%%%%%%%%%%%%%%%%%%%%%%%%%%%%%%%%%%%%%%%%%%%%
%%% 
%%% SUPPLEMENTARY MATERIALS
%%%
%%%%%%%%%%%%%%%%%%%%%%%%%%%%%%%%%%%%%%%%%%%%%%%%%%%%%%%%%%%%%%%%%%%%%%%%%%%%%%%%%%%%%%%%%%%%%%%%%%%%
%%%%%%%%%%%%%%%%%%%%%%%%%%%%%%%%%%%%%%%%%%%%%%%%%%%%%%%%%%%%%%%%%%%%%%%%%%%%%%%%%%%%%%%%%%%%%%%%%%%%

\newpage
\setcounter{page}{1}
\begin{center}
{\bf \Large  Supplementary Material\\for\\``Decomposing Treatment Effect Variation''}
\end{center}

    \setcounter{equation}{0}
        \setcounter{section}{0}
    \setcounter{lemma}{0}
    \setcounter{theorem}{0}
    \setcounter{figure}{0}
        \setcounter{table}{0}
        \setcounter{corollary}{0}

\renewcommand{\theequation}{A.\arabic{equation}}
\renewcommand{\thelemma}{A.\arabic{lemma}}
\renewcommand{\thecorollary}{A.\arabic{corollary}}
\renewcommand{\thesection}{Appendix~\Alph{section}}
\renewcommand{\thetheorem}{A.\arabic{theorem}}
\renewcommand{\thefigure}{A.\arabic{figure}}
\renewcommand{\thetable}{A.\arabic{table}}

\medskip 
 \ref{appendix::a} gives all the proofs and \ref{appendix::b} provides the additional commentary mentioned in the main text. The finite population central limit theorem (FPCLT) we use for our asymptotic proofs is Theorem 5 of \citet{li2017general}, which requires some mild moment conditions on the covariates and potential outcomes, as outlined in the main text.

\section{Lemmas and Proofs}
\label{appendix::a}

Before we prove Theorem~\ref{thm::general-neyman}, we provide a few lemmas to ease the notational burden and amount of algebra of subsequent calculations.
These lemmas allow us to derive expressions for our estimators in terms of matrix algebra rather than the summation-style approach typically seen for Neyman-style derivations in the literature.

To begin, let $\bone_n = (1, \ldots, 1)^\T$ and $\bzero_n= (0, \ldots, 0)^\T$ be column vectors of length $n$, and $\bI_n$ be the $n \times n$ identity matrix. Then 
$
\bS_n = \bI_n  - n^{-1}\bone_n \bone_n^\T
$
is the projection matrix orthogonal to $\bone_n$ with $ \bS_n \bone_n= \bzero_n$. 
Under this formulation, the covariance matrix of the treatment assignment vector is a scaled projection matrix orthogonal to $\bone_n$, as shown in the following lemma.

\begin{lemma}\label{lemma::cov}
The treatment assignment vector $\bT$ of a completely randomized experiment has 
$$
E(\bT) = \frac{n_1}{n} 1_n, \quad \cov(\bT) = \frac{n_1n_0}{n(n-1)} \bS_n.
$$
\end{lemma}

\begin{proof}[Proof of Lemma \ref{lemma::cov}]
The conclusions follow from 
$$
E(T_i) = \frac{n_1}{n}, \quad \var(T_i) = \frac{n_1n_0}{n^2}, \quad \cov(T_i, T_j) = - \frac{n_1n_0}{  n^2(n-1) },\quad  (i\neq j).
$$
\end{proof}

The projection matrix $\bS_n$ acts as a covariance operator as illustrated by the following lemma.

\begin{lemma}\label{lemma::s-cov}
Let $\bU_i, \bV_i \in \mathbb{R}^K$ be column vectors of length $K$. Define $\mathcal{U} = [\bU_1, \bU_2, \ldots, \bU_n]$ and $\mathcal{V} = [\bV_1, \bV_2, \ldots, \bV_n] \in \mathbb{R}^{K\times n}
$ as two matrices of dimension $K \times n$. 
If $
\bar{\bU} = n^{-1} \sumn \bU_i = n^{-1} \mathcal{U} \bone$ and $ \bar{\bV} = n^{-1} \sumn \bV_i = n^{-1} \mathcal{V} \bone,
$
then
\begin{eqnarray}
\label{eq::cov-operator}
  \mathcal{U} \bS_n \mathcal{V} ^\T       =  \sumn  (  \bU_i - \bar{\bU}   )  (\bV_i - \bar{\bV})^\T .
\end{eqnarray}
In particular, when $\bU_i=\bV_i$, 
$$
\mathcal{V} \bS_n \mathcal{V}^\T =   \sumn  (  \bV_i - \bar{\bV}   )  (\bV_i - \bar{\bV})^\T = (n-1) \mathcal{S}(\bV) .
$$
\end{lemma}

\begin{proof}
[Proof of Lemma \ref{lemma::s-cov}]
The left hand side of \eqref{eq::cov-operator} is equal to 
\begin{eqnarray*}
  \mathcal{U} \bS_n \mathcal{V}^\T  =  \mathcal{U}  \mathcal{V}^\T - 
 n^{-1}\left(  \mathcal{U} \bone_n \right) \left(  \mathcal{V} \bone_n \right )^\T
  = \sumn \bU_i \bV_i^\T - n^{-1}  (n\bar{\bU})(n\bar{\bV})^\T
  = \sumn \bU_i \bV_i^\T - n\bar{\bU} \bar{\bV}^\T,
\end{eqnarray*}
which is the same as the right hand side of \eqref{eq::cov-operator}.
\end{proof}

\paragraph{Theorem~\ref{thm::general-neyman}: A generalized, vector-outcome version of Neyman.}
To prove the generalized Neyman result, we bundle our vector potential outcomes into matrices and use the above lemmas to obtain their covariance matrix. The theorem is exact, no asymptotics. Using the FPCLT to show that the estimator has an approximately Normal distribution, allowing for classic testing and inference, is a separate, subsequent step.

\begin{proof}
[Proof of Theorem \ref{thm::general-neyman}]
Define 
$
\mathcal{V}_1 = [ \bV_1(1), \ldots, \bV_n(1) ]$ and $ \mathcal{V}_0 = [ \bV _1(0), \ldots, \bV_n(0)  ]
$
as the matrices of the potential outcomes.
Then the Neymanian simple difference in means estimator has the following representation:
\begin{eqnarray*}
\widehat{ \btau }_\bV &=&\bar{\bV}_1^{\obs} - \bar{\bV}_0^{\obs} \\
&=& \frac{1}{n_1} \sumn T_i  \bV_i(1) - \frac{1}{n_0} \sumn (1-T_i) \bV _i(0)\\
&=& \frac{1}{n_1} \mathcal{V}_1 \bT - \frac{1}{n_0} \mathcal{V}_0 \left( \bone - \bT \right) \\
&=& \left( \frac{\mathcal{V}_1}{n_1} + \frac{\mathcal{V}_0}{n_0} \right) \bT - \frac{1}{n_0} \mathcal{V}_0 \bone .
%\sumn T_i \left\{   \frac{ \bV_i(1)}{n_1} + \frac{  \bV_i(0)}{n_0} \right\}  - \frac{1}{n_0} \sumn \bV_i(0).
\end{eqnarray*}

Now the unbiasedness of $ \widehat{ \btau }_\bV $ follows from the linearity of the expectation and Lemma \ref{lemma::cov}.  
For the covariance, note the second term in the above is constant, and so is not involved.
Applying Lemmas \ref{lemma::cov} and \ref{lemma::s-cov}, we can obtain the covariance matrix of $\widehat{  \btau  }_\bV $:
\begin{eqnarray*}
\cov (  \widehat{  \btau  }_\bV  ) 
&=&  \left( \frac{\mathcal{V}_1}{n_1} + \frac{\mathcal{V}_0}{n_0} \right) 
 \cov(\bT) 
\left( \frac{\mathcal{V}_1}{n_1} + \frac{\mathcal{V}_0}{n_0} \right) ^\T \\
&=&  \frac{n_1 n_0}{n (n-1)}   \left( \frac{\mathcal{V}_1}{n_1} + \frac{\mathcal{V}_0}{n_0} \right) 
\bS_n
\left( \frac{\mathcal{V}_1}{n_1} + \frac{\mathcal{V}_0}{n_0} \right) ^\T \\
&=&  \frac{n_1 n_0}{n (n-1)}   \left(    \frac{1}{n_1^2}  \mathcal{V}_1 \bS_n \mathcal{V}_1^\T  
+ \frac{1}{n_0^2} \mathcal{V}_0 \bS_n \mathcal{V}_0 ^\T
+ \frac{1}{n_1n_0} \mathcal{V}_0 \bS_n \mathcal{V}_1^\T
+  \frac{1}{n_1 n_0} \mathcal{V}_1 \bS_n \mathcal{V}_0^\T
\right)\\
&=& \frac{n_0}{nn_1} \mathcal{S} \{  \bV(1) \} + \frac{n_1}{nn_0} \mathcal{S} \{  \bV(0) \}
+ \frac{1}{n(n-1)}  (\mathcal{V}_0 \bS_n \mathcal{V}_1^\T + \mathcal{V}_1 \bS_n \mathcal{V}_0 ^\T ).
\end{eqnarray*}
%\cmntM{Isn't the third term above some sort of covariance, and can we simplify if we look at it in those terms?}
To simplify the third term, we use the fact $ \ba \bb^\T + \bb \ba^\T = \ba\ba^\T  + \bb\bb^\T - (\ba-\bb)(\ba-\bb)^\T$ for two column vectors $\ba$ and $\bb$, we have
\begin{eqnarray*}
&&\{   \bV_i(1) - \bar{ \bV}(1)    \}  \{   \bV_i(0) -   \bar{\bV}(0)    \} ^\T 
+ \{   \bV_i(0) - \bar{ \bV}(0)    \}  \{   \bV_i(1) -   \bar{  \bV}(1)    \} ^\T \\
&=& \{   \bV_i(1) - \bar{\bV} (1)    \} \{   \bV_i(1) - \bar{ \bV}(1)    \} ^\T
+ \{   \bV_i(1) - \bar{\bV}(1)   \} \{   \bV_i(1) - \bar{ \bV}(1)    \} ^\T   \\
&&
- \{   \bV_i(1) - \bV_i(0)  - \bar{\bV}(1)   +  \bar{ \bV}(0)  \}
\{   \bV_i(1) - \bV_i(0)  - \bar{\bV}(1)   +  \bar{ \bV}(0)  \} ^\T .
\end{eqnarray*}
Summing over $i=1, \ldots, n$ and applying Lemma \ref{lemma::s-cov}, we have
$$
\frac{ \mathcal{V}_0 \bS_n \mathcal{V}_1 ^\T  }{n-1} + \frac{  \mathcal{V}_1 \bS_n \mathcal{V}_0 ^\T }{n-1}
= \mathcal{S} \{  \bV(1) \} + \mathcal{S} \{  \bV(0) \}   - \mathcal{S}\{ \bV(1) - \bV(0) \}. 
$$
Therefore, the covariance of $ \widehat{ \btau }_\bV$ can be simplified as:
\begin{eqnarray*}
\cov(\widehat{ \btau }_\bV) 
&=&  
\frac{n_0}{nn_1} \mathcal{S} \{  \bV(1) \} + \frac{n_1}{n n_0} \mathcal{S} \{  \bV(0) \}
+ \frac{1}{n}\left[   \mathcal{S} \{  \bV(1) \} + \mathcal{S} \{  \bV(0) \}   - \mathcal{S}\{ \bV(1) - \bV(0) \} \right]\\
&=&\frac{   \mathcal{S} \{  \bV(1) \}   }{ n_1 } + 
 \frac{   \mathcal{S} \{  \bV(0) \}   }{ n_0 }  - 
  \frac{   \mathcal{S} \{  \bV(1) - \bV(0) \}   }{n } . 
\end{eqnarray*}
\end{proof}

\paragraph{Theorem~\ref{thm::randomization-inference}: Behavior of $\widehat{\bbeta}_{\RI}$.}
To show properties of $\widehat{\bbeta}_{\RI}$ we express the systematic variation as a vector of new potential outcomes of the original outcome scaled by the different covariates of interest.
This allows for immediate use of Theorem~\ref{thm::general-neyman}.

\begin{proof}
[Proof of Theorem \ref{thm::randomization-inference}.]
Because $\widehat{\bS }_{xt}$ is the sample mean for $\{ \bX_i Y_i^\obs   : T_i = t, i=1,\ldots, n\} = \{ \bX_i Y_i(t)   : T_i = t, i=1,\ldots, n\}$, it is unbiased for the population mean $\bS_{xt}$. Thus, the estimator
$
\widehat{\bbeta}_{\RI} 
$
is also unbiased for $\bbeta$ as $\bS_{xx}^{-1}$ is fixed and the expectation is linear. Its sampling covariance over all possible randomizations is
$$
\cov(   \widehat{ \bbeta }_{\RI} ) = \bS _{xx}^{-1}  \cov(  \widehat{ \bS }_{x1} - \widehat{ \bS }_{x0} ) \bS_{xx}^{-1} .
$$
Therefore, we need only to obtain the covariance of
$$  
\widehat{ \bS }_{x1} - \widehat{ \bS }_{x0}  =  \frac{1}{n_1}   \sumn T_i  \bX_iY_i^{\obs } -  
 \frac{1}{n_0}   \sumn (1-T_i) \bX_iY_i^{\obs },
$$ 
which is the difference between the sample means of $\left\{  \bX_i Y_i(1)  : i=1,\ldots, n   \right\}$ and $\left\{  \bX_iY_i(0)  : i=1,\ldots,N   \right\}$ under treatment and control. Viewing $\bX_i Y_i^\obs$ as a vector outcome in a completely randomized experiment, we can apply Theorem \ref{thm::general-neyman} to obtain  
$$
\cov(\widehat{\bS}_{x1} - \widehat{\bS}_{x0}) =\frac{  \S\{\bX   Y(1)  \} }{ n_1} +\frac{  \S\{  \bX   Y(0) \} } { n_0  } -  \frac{ \S( \bX \tau ) } { n},
$$
which completes the proof.
\end{proof}

\paragraph{Theorem~\ref{theorem::ols}: Behavior of $\widehat{\bbeta}_{\OLS}$.}

We first use the well-known fact that the estimate from a OLS model with treatment fully interacted with covariates is equivalent to separate regressions of outcome onto covariates for the control and treatment groups.
This means we can obtain $\widehat{\bgamma}_\OLS$ by running a regression of $Y^{\obs}$ onto $\bX$ using the control group data, and $\widehat{\bgamma + \bbeta}_\OLS$ by running regression of $Y^{\obs}$ onto $\bX$ using the treatment group data, giving estimated coefficients of
\[ \widehat{\bgamma}_{\OLS} = \widehat{\bS}_{xx,0}^{-1}  \widehat{\bS}_{x0} \]
and 
\[ \widehat{\bbeta}_{\OLS} = \widehat{\bS}_{xx,1}^{-1}  \widehat{\bS}_{x1} -  \widehat{\bS}_{xx,0}^{-1}  \widehat{\bS}_{x0} . \]
As a quick heuristic argument for this, consider that the maximization problem for the interacted model will separate into two components, one for each group.  Then re-parameterize to get the above.

We now prove the properties of $\widehat{\bbeta}_{\OLS}$.
Here we have to use asymptotics for the entire theorem, unlike the case of $\widehat{\bbeta}_{\RI}$, where the mean and covariance are exact and the asymptotics are only needed for the asymptotic normality of the estimator.

\begin{proof}[Proof of Theorem \ref{theorem::ols}.]
First expand the difference of $\widehat{\bbeta}_\OLS$ and $\bbeta$ as
$$
\widehat{\bbeta}_\OLS - \bbeta =   \widehat{\bS}_{xx,1}^{-1}  ( \widehat{\bS}_{x1} -  \widehat{\bS}_{xx,1}  \bgamma_1 ) 
-  \widehat{\bS}_{xx,0}^{-1}  ( \widehat{\bS}_{x0} -  \widehat{\bS}_{xx,0}  \bgamma_0 ) ,
$$
This will be close to the related quantity of
\begin{eqnarray}
\Delta =   \bS_{xx}^{-1}  ( \widehat{\bS}_{x1} -  \widehat{\bS}_{xx,1} \gamma_1 ) 
-  \bS_{xx}^{-1}  ( \widehat{\bS}_{x0} -  \widehat{\bS}_{xx,0} \bgamma_0 ) .
\label{eq::ols-tilde} 
\end{eqnarray}
For the above to make sense and hold, we here need our asymptotic framework.
In particular, we need the associated moment conditions described in the main text.
We next observe that the difference between $\widehat{\bbeta}_\OLS - \bbeta$ and $\Delta$ is of higher order, because
\begin{eqnarray}
(\widehat{\bbeta}_\OLS - \beta) - \Delta
&=&
( \widehat{\bS}_{xx,1}^{-1} -  \bS_{xx}^{-1} )  ( \widehat{\bS}_{x1} -  \widehat{\bS}_{xx,1}  \bgamma_1 ) 
-
( \widehat{\bS}_{xx,0}^{-1}   -  \bS_{xx}^{-1} )  ( \widehat{\bS}_{x0} -  \widehat{\bS}_{xx,0} \bgamma_0 )  \label{eq::fourterms} \\
&=& O_P(n^{-1/2}) O_P(n^{-1/2}) - O_P(n^{-1/2}) O_P(n^{-1/2})  = O_P(n^{-1}), \label{eq::higher-order}
\end{eqnarray}
following from the FPCLT for the four terms in \eqref{eq::fourterms}. This is an argument commonly used in the survey sampling literature for ratio estimators \citep{cochran::1977}.

We next focus on the asymptotic distribution of $\Delta$, because the asymptotic distribution of $\widehat{\bbeta}_\OLS - \bbeta$ will be the same.
Further simplify \eqref{eq::ols-tilde} as
\begin{eqnarray}\label{eq::representation}
\Delta = 
 \bS_{xx}^{-1}  \left[ \frac{1}{n_1} \sum_{i=1}^n T_i \bX_i e_i(1)  - \frac{1}{n_0} \sum_{i=1}^n (1-T_i) \bX_i e_i(0)   \right],
\end{eqnarray}
where $e_i(1) = Y_i(1) - \bX_i^\T \bgamma_1$ and $e_i(0) = Y_i(0) - \bX_i^\T \bgamma_0$ are the residual potential outcomes.
(To see the above, note, for example, that both $\widehat{\bS}_{x1}$ and $\widehat{\bS}_{xx,1}$ are sums over the treatment units, and we can factor out an $\bX_i$ to get $\bX_i$ times the difference in the $Y_i$ and predicted $Y_i$.)

%why was the following clause here?  Needed?  -luke   -->
%, satisfying $e_i(1) - e_i(0) = \varepsilon_i$ and thus $ \sumn \bX_i \{ e_i(1) - e_i(0) \}/n = \bS_{x\varepsilon} = 0.$ 
Applying Theorem \ref{thm::general-neyman} to the vector outcome $\bX e$, we obtain the covariance matrix of $\Delta$, to which $\widehat{\bbeta}_\OLS - \beta$ converges to due to \eqref{eq::higher-order}.
The asymptotic normality follows from the representation \eqref{eq::representation} and the FPCLT. 
\end{proof}

\paragraph{Theorem~\ref{thm::bounds_for_s_tau_tau}: Bounds for $R^2_\tau$.}
To prove Theorem \ref{thm::bounds_for_s_tau_tau}, we need to invoke the following Fr\'echet--Hoeffding inequality \citep{hoeffding::1941, frechet::1951,heckman::1997, aronow::2014}.

\begin{lemma}
\label{lemma::frechet-hoeffding}
If we know only the marginal distributions of two random variables $X\sim F_X(x)$ and $Y\sim F_Y(y)$, then $E(XY)$ can be sharply bounded by
$$
\int_0^1 F^{-1}_X(u) F^{-1}_Y(1-u) \d u 
\leq E(XY) \leq 
\int_0^1 F^{-1}_X(u) F^{-1}_Y(u) \d u.
$$
\end{lemma}

Lemma \ref{lemma::frechet-hoeffding} immediately implies the following bound for $\var(X-Y)$ if $E(X-Y)=0$.

\begin{lemma}
\label{lemma::variance-bound}
If we know only the marginal distributions $X\sim F_X(x), Y\sim F_Y(y)$ and $E(X-Y)=0$, then $\var(X-Y)$ can be sharply bounded by
$$
\int_0^1 \{ F_X^{-1}(u) - F_Y^{-1}(u)   \}^2 \d u
\leq \var(X-Y) \leq
\int_0^1 \{ F_X^{-1}(u) - F_Y^{-1}(1-u)   \}^2 \d u
$$
\end{lemma}

\begin{proof}
[Proof of Lemma \ref{lemma::variance-bound}.]
The variance $\var(X-Y)$ can be decomposed as
\[ \var(X-Y) = E(X-Y)^2 = E(X^2)+E(Y^2)-2E(XY), \]
which depends on the following three terms:
\begin{eqnarray*}
E(X^2) =& \int x^2   \d F_X(x) &= \int_0^1 \{  F_X^{-1}(u)] \} ^2 \d u,\\
E(Y^2) =&  \int_0^1 \{  F_Y^{-1}(u)\} ^2 \d u &= \int_0^1  \{  F_Y^{-1}(1-u) \}^2 \d u,\\
\int_0^1 F^{-1}_X(u) F^{-1}_Y(1-u) \d u 
\leq& E(XY) &\leq 
\int_0^1 F^{-1}_X(u) F^{-1}_Y(u) \d u.
\end{eqnarray*}
Plug the above expressions into the variance of $X-Y$ to obtain the desired bounds.
\end{proof}

Applying Lemma \ref{lemma::variance-bound}, we can easily prove Theorem \ref{thm::bounds_for_s_tau_tau}.

\begin{proof}
[Proof of Theorem \ref{thm::bounds_for_s_tau_tau}.]
Because $S_{\tau\tau} = S_{\delta\delta} + S_{\varepsilon\varepsilon}$, we need only to bound $S_{\varepsilon\varepsilon}$, which is the finite population variance of 
$
\varepsilon_i = \{  Y_i(1) - \bX_i^\T\bgamma_1 \} -  \{  Y_i(0)  - \bX_i^\T \bgamma_0 \}  = e_i(1) - e_i(0). 
$ 
We can identify the marginal distributions of $\{   e_i(1)  : i=1,\ldots, n\} $ and $\{ e_i(0)   : i=1,\ldots, n\}$, and also know that $n^{-1}\sumn \varepsilon_i = 0$. Therefore, the bounds in Lemma \ref{lemma::variance-bound} imply the bounds in Theorem \ref{thm::bounds_for_s_tau_tau}.
\end{proof}

\paragraph{Theorem~\ref{thm::sensitivity analysis}: Sensitivity analysis.}

\begin{proof}
[Proof of Theorem \ref{thm::sensitivity analysis}.]
The joint distribution of $(U_1, U_0)$ is 
\begin{eqnarray*}
C(u_1, u_0) &=& P(U_1\leq u_1, U_0\leq u_0) \\
&=& \rho P(U_0\leq u_1, U_0\leq u_0) + (1-\rho) P(V_0\leq u_1, U_0\leq u_0) \\
&=& \rho \min(u_1, u_0) + (1-\rho)u_1 u_0.
\end{eqnarray*}
Therefore, the distribution function $C(u_1, u_0) $ is a weighted average of $\min(u_1,u_0) = C_R(u_1, u_0)$ and $u_1u_0 = C_I(u_1,u_0)$, i.e., the joint distributions when $U_1=U_0$ and $U_1\indep U_0$, respectively.

According to \citet[][Theorem 5.1.6]{nelsen::2007}, Spearman's rank correlation coefficient between $e(1)$ and $e(0)$ is
\begin{eqnarray*}
12 \int_{0}^1 \int_0^1 \{ C(u_1, u_0) - u_1 u_0\}  \d u_1 \d u_0&=& 12\rho \int_{0}^1 \int_0^1 \{  \min(u_1, u_0) - u_1u_0  \} \d u_1 \d u_0 \\
&=& 12\rho \left(   2\int_0^1 \d u_1\int_0^{u_1} u_0 \d u_0  - \frac{1}{4}  \right) \\
&=&12\rho (1/3 - 1/4) = \rho.
\end{eqnarray*}

To complete the proof of the theorem, we need only to show that the covariance between $e(1)$ and $e(0)$ is linear in $\rho$, which follows from
\begin{eqnarray*}
&& \int_0^1 \int_0^1 F_1^{-1}(u_1) F_0^{-1}(u_0) \d C(u_1,u_0)  \\
&=& \rho  \int_0^1 \int_0^1 F_1^{-1}(u_1) F_0^{-1}(u_0)  \d C_R(u_1,u_0)
+\rho  \int_0^1 \int_0^1 F_1^{-1}(u_1) F_0^{-1}(u_0)  \d C_I(u_1,u_0) \\
&=& \rho  \int_0^1  F_1^{-1}(u) F_0^{-1}(u) \d u + 
(1-\rho )\int_0^1  F_1^{-1}(u)  \d u\int_0^1   F_0^{-1}(u) \d u.
\end{eqnarray*}
\end{proof}

\paragraph{Theorem~\ref{thm::randomization-inference-noncompliance}: Extending to non-compliance.}

Theorem~\ref{thm::randomization-inference-noncompliance} shows how to estimate the outcome-to-covariate relationships of the Compliers by estimating different aggregate covariance relationships across all the strata for different observed groups and then taking differences.
Due to the exclusion restriction for the Never Takers and Always Takers, this gives our desired relationships for the Compliers only.

First, a small bit of notation of, due to the exclusion restrictions for Never Takers and Always Takers, defining the population covariance between $\bX$ and $Y(1)=Y(0)$ within stratum $U=a$ and $U=n$ as
$$
\bS_{x.,u} = \frac{1}{n_u} \sum_{i=1}^n  I_{(U_i=u)} \bX_i Y_i(1)
= \frac{1}{n_u} \sum_{i=1}^n  I_{(U_i=u)} \bX_i Y_I(0), \quad (u=a,n).
$$

\begin{proof}
[Proof of Theorem \ref{thm::randomization-inference-noncompliance}.]
We first create an estimator for $\bS_{xx,c}$.
From the observed data with $(T_i,D_i)=  (1,1)$, we have
\begin{eqnarray}
E\left\{  \frac{1}{n_1} \sum_{i=1}^n  T_i D_i  \bX_i \bX_i^\T     \right\}  
&=&E\left\{  \frac{1}{n_1}  \sum_{i=1}^n  T_i I_{(U_i=a)} \bX_i \bX_i^\T    
+ \frac{1}{n_1}  \sum_{i=1}^n  T_i I_{(U_i=c)} \bX_i \bX_i^\T    \right\}   \nonumber \\
&=& \pi_a \bS_{xx,a} + \pi_c \bS_{xx,c} . 
\label{eq::xx-11}
\end{eqnarray}
Similar to (\ref{eq::xx-11}), we have 
\begin{eqnarray}
E\left\{   \frac{1}{n_1} \sum_{i=1}^n T_i(1-D_i) \bX_i  \bX_i^\T     \right\} &=& \pi_n  \bS_{xx, n}  , 
\label{eq::xx-10}
\\
E\left\{    \frac{1}{n_0}    \sum_{i=1}^n (1-T_i) D_i  \bX_i \bX_i^\T    \right\} &=& \pi_a  \bS_{xx, a}  ,
\label{eq::xx-01}
\\
E\left\{   \frac{1}{n_0}    \sum_{i=1}^n  (1-T_i)(1-D_i) \bX_i \bX_i^\T     \right\} &=& \pi_n \bS_{xx,n} + \pi_c \bS_{xx,c} .
\label{eq::xx-00}
\end{eqnarray}
Subtracting the left sides of (\ref{eq::xx-01}) from (\ref{eq::xx-11}), or subtracting the left sides of (\ref{eq::xx-10}) from (\ref{eq::xx-00}), give unbiased estimators for $\pi_c \bS_{xx,c}.$

Second, analogous to the $\bS_{xx,c}$, we consider the sample covariances between $\bX$ and $Y^\obs$ to obtain estimators for $\bS_{x1,c}$ and $\bS_{x0,c}$.
From the observed data with $(T_i,D_i)=  (1,1)$, we have 
\begin{eqnarray}
E\left\{  \frac{1}{n_1}    \sum_{i=1}^n T_i D_i  \bX_i Y_i^\obs    \right\}  
&=& E\left\{  \frac{1}{n_1}   \sum_{i=1}^n T_i I_{(U_i=a)}  \bX_i Y_i(1)  
+ \frac{1}{n_1}    \sum_{i=1}^n T_i I_{(U_i=c)}  \bX_i Y_i(1)    \right\}
\nonumber \\
&=& \pi_a \bS_{x.,a} + \pi_c \bS_{x1,c}.
\label{eq::xy-11}
\end{eqnarray}
Similar to (\ref{eq::xy-11}), we have 
\begin{eqnarray}
E\left\{   \frac{1}{n_1}   \sum_{i=1}^n T_i(1-D_i) \bX_i Y_i^\obs    \right\} = \pi_n  \bS_{x., n}  , \label{eq::xy-10}
\\
E\left\{    \frac{1}{n_0}    \sum_{i=1}^n (1-T_i) D_i  \bX_i   Y_i^\obs     \right\} = \pi_a  \bS_{x., a}  ,  \label{eq::xy-01}
\\
E\left\{    \frac{1}{n_0}  \sum_{i=1}^n (1-T_i)(1-D_i) \bX_i    Y_i^\obs     \right\} = \pi_n \bS_{x.,n} + \pi_c \bS_{x0,c}  .  \label{eq::xy-00}
\end{eqnarray}
Subtracting (\ref{eq::xy-01}) from (\ref{eq::xy-11}), and subtracting (\ref{eq::xy-10}) from (\ref{eq::xy-00}), we obtain the results in (\ref{eq::unbiased-xy}).
\end{proof}

\paragraph{Corollary~\ref{coro::ri-nomcompliance}: Behavior of $\widehat{\bbeta}_{c,\RI}$.}
Theorem~\ref{thm::randomization-inference-noncompliance} shows how to obtain unbiased estimates of the components of our estimator, which we can then plug in to obtain a consistent estimator of $\bbeta_c$.
We next show how this plug-in estimator behaves.

\begin{proof}[Proof of Corollary \ref{coro::ri-nomcompliance}]

First we write
\begin{eqnarray*}
\widehat{\bbeta}_{c,\RI} - \bbeta_c &=& 
(\widehat{\bS}_{xx,11} - \widehat{\bS}_{xx,01})^{-1}    \{   \widehat{\bS}_{x1,11} - \widehat{\bS}_{x0,01}  -  (\widehat{\bS}_{xx,11} - \widehat{\bS}_{xx,01}) \bgamma_{1c}  \} \\
&& -
 (\widehat{\bS}_{xx,00} - \widehat{\bS}_{xx,10})^{-1}    \{   \widehat{\bS}_{x0,00} - \widehat{\bS}_{x1,10}  -   (\widehat{\bS}_{xx,00} - \widehat{\bS}_{xx,10}) \bgamma_{0c}  \},
\end{eqnarray*}
second we introduce
\begin{eqnarray*}
\Delta_c &=& 
(  \pi_{c} \bS_{xx,c}  )^{-1}    \{   \widehat{\bS}_{x1,11} - \widehat{\bS}_{x0,01}  -  (\widehat{\bS}_{xx,11} - \widehat{\bS}_{xx,01}) \bgamma_{1c}  \} \\
&&-
(  \pi_{c} \bS_{xx,c}  )^{-1}     \{   \widehat{\bS}_{x0,00} - \widehat{\bS}_{x1,10}  -   (\widehat{\bS}_{xx,00} - \widehat{\bS}_{xx,10}) \bgamma_{0c}  \},
\end{eqnarray*}
third we observed that the difference between $\widehat{\bbeta}_{c,\RI} - \bbeta_c$ and $\Delta_c$ has higher order following the same argument as \eqref{eq::higher-order}. 
Therefore, we need only to find the asymptotic distribution of $\Delta_c$.

Simple algebra gives
\begin{eqnarray*}
\Delta_c &=&
(  \pi_{c} \bS_{xx,c}  )^{-1}  \Big [ 
\frac{1}{n_1} \sumn T_iD_i \bX_i Y_i(1) - \frac{1}{n_0} \sumn (1-T_i)D_i \bX_iY_i(0)  \\
&& - \frac{1}{n_1} \sumn T_i D_i \bX_i \bX_i^\T \bgamma_{c1} + \frac{1}{n_0} \sumn (1-T_i)D_i \bX_i \bX_i^\T \bgamma_{c1} \\
&& - \frac{1}{n_0} \sumn (1-T_i)(1-D_i) \bX_i Y_i(0) + \frac{1}{n_1}\sumn T_i(1-D_i)\bX_i Y_i(1) \\
&& + \frac{1}{n_0} \sumn (1-T_i)(1-D_i) \bX_i \bX_i^\T \bgamma_{c0} - \frac{1}{n_1} \sumn T_i(1-D_i) \bX_i \bX_i^\T \bgamma_{c0}
\Big]\\
&=&
(  \pi_{c} \bS_{xx,c}  )^{-1}  \Big [ 
\frac{1}{n_1} \sumn T_i  I_{(U_i=a)} \bX_i Y_i(1)  + \frac{1}{n_1} \sumn T_i  I_{(U_i=c)} \bX_i Y_i(1) - \frac{1}{n_0} \sumn (1-T_i)I_{(U_i=a)} \bX_iY_i(0)  \\
&& - \frac{1}{n_1} \sumn T_i I_{(U_i=a)}  \bX_i \bX_i^\T \bgamma_{c1} - \frac{1}{n_1} \sumn T_i I_{(U_i=c)}  \bX_i \bX_i^\T \bgamma_{c1} 
+ \frac{1}{n_0} \sumn (1-T_i) I_{(U_i=a)} \bX_i \bX_i^\T \bgamma_{c1} \\
&& - \frac{1}{n_0} \sumn (1-T_i) I_{(U_i=n)} \bX_i Y_i(0) - \frac{1}{n_0} \sumn (1-T_i) I_{(U_i=c)} \bX_i Y_i(0)  
+  \frac{1}{n_1}\sumn T_i I_{(U_i=n)} \bX_i Y_i(1) \\
&& + \frac{1}{n_0} \sumn (1-T_i) I_{(U_i=n)} \bX_i \bX_i^\T \bgamma_{c0}  + \frac{1}{n_0} \sumn (1-T_i) I_{(U_i=c)} \bX_i \bX_i ^\T \bgamma_{c0}
- \frac{1}{n_1} \sumn T_i I_{(U_i=n)} \bX_i \bX_i^\T \bgamma_{c0}
\Big]\\
&=& 
(  \pi_{c} \bS_{xx,c}  )^{-1}  \Big \{ 
\frac{1}{n_1} \sumn T_i X_i   \left[ I_{(U_i=a)} ( Y_i(1) - \bX_i^\T \bgamma_{c1}  ) +  I_{(U_i=n)}  (Y_i(1) - \bX_i ^\T \bgamma_{c0}  ) + I_{(U_i=c)} ( Y_i(1) - \bX_i^\T \bgamma_{c1}  ) \right]  \\
&& - \frac{1}{n_0} \sumn (1-T_i) X_i   \left [ I_{(U_i=a)} ( Y_i(0) - \bX_i^\T \bgamma_{c1}  ) +  I_{(U_i=n)}  (Y_i(0) - \bX_i ^\T \bgamma_{c0}  ) + I_{(U_i=c)} ( Y_i(0) - \bX_i^\T \bgamma_{c0}  ) \right]  \Big\}  .
\end{eqnarray*}
According to the definitions of the residual potential outcomes $e_i'(1)$ and $e_i'(0)$ in the main text, the above formula reduces to 
\begin{eqnarray}\label{eq::represent-ri-c}
\widetilde{\beta}_{c,\RI} - \beta_c 
= (  \pi_{c} \bS_{xx,c}  )^{-1} \left[  
\frac{1}{n_1} \sumn T_i   \bX_i e_i ' (1) - \frac{1}{n_0} \sumn (1-T_i)   \bX_i e_i' (0)  \right].
\end{eqnarray} 
The representation in \eqref{eq::represent-ri-c} implies the asymptotic covariance matrix according to Theorem \ref{thm::general-neyman} and the asymptotic normality of $\widetilde{\bbeta}_{c,\RI}$ according to the FPCLT. 
\end{proof}

\paragraph{Theorem~\ref{thm::TSLS}: Behavior of $\widehat{\bbeta}_\TSLS$.}

While the amount of notation and matrix algebra is considerably more in scope, the overall structure of the proof follows the earlier one for the OLS estimator for the ITT.
In particular, we show the estimator asymptotically converges to a more tractable version that has a fixed portion, and then use the usual covariance argument on the remaining terms.
Before doing this, we first show the probability limits of the estimator by working through the matrix algebra.

\begin{proof}
[Proof of Theorem \ref{thm::TSLS}.]
First, we find the probability limits of the TSLS estimators:
\begin{eqnarray}
\begin{pmatrix}
\widehat{\bgamma}_\TSLS\\
\widehat{\bbeta}_\TSLS
\end{pmatrix}
&=&
\left\{  \frac{1}{n}   \sum_{i=1}^n
\begin{pmatrix}
\bX_i\\
T_i \bX_i
\end{pmatrix}
(\bX_i^\T, D_i \bX_i^\T)
\right\} ^{-1}
\left\{ \frac{1}{n} \sum_{i=1}^n
\begin{pmatrix}
\bX_i\\
T_i \bX_i
\end{pmatrix}
Y_i^\obs
\right\}  \nonumber  \\
&=&
\begin{pmatrix}
n^{-1} \sum_{i=1}^n \bX_i \bX_i^\T & n^{-1} \sum_{i=1}^n D_i \bX_i \bX_i^\T\\
n^{-1} \sum_{i=1}^n T_i \bX_i \bX_i^\T & n^{-1} \sum_{i=1}^n T_i D_i \bX_i \bX_i^\T
\end{pmatrix}^{-1}
\begin{pmatrix}
n^{-1} \sum_{i=1}^n \bX_i Y_i^\obs\\
n^{-1} \sum_{i=1}^n T_i \bX_i Y_i^\obs 
\end{pmatrix}  \nonumber \\
&\stackrel{P}{\longrightarrow} &
\begin{pmatrix}
\bA&\bB\\
\bC&\bD
\end{pmatrix}^{-1}
\begin{pmatrix}
\bG\\
\bH
\end{pmatrix}
.\label{eq::ABCDGH}
\end{eqnarray}
The above term $\bA$ is $\bA= \bS_{xx}$, and terms $(\bB,\bC,\bD,\bG,\bH)$ are the population limits of the sample quantities. 
We will find each of them. 
Term $\bB$ is
\begin{eqnarray*}
\bB&=& E\left\{   \frac{1}{n} \sum_{i=1}^n D_i \bX_i \bX_i^\T  \right\} 
= E\left\{   \frac{1}{n} \sum_{i=1}^n T_i D_i \bX_i \bX_i^\T  + \frac{1}{n} \sum_{i=1}^n (1-T_i) D_i \bX_i \bX_i^\T   \right\} \\
&=& E\left\{  \frac{1}{n} \sum_{i=1}^n T_i I_{(U_i = a)} \bX_i \bX_i^\T 
+ \frac{1}{n} \sum_{i=1}^n T_i I_{(U_i = c)} \bX_i \bX_i^\T 
+ \frac{1}{n} \sum_{i=1}^n (1-T_i) I_{(U_i = a)} \bX_i \bX_i^\T   \right\} \\
&=& p_1 \pi_a \bS_{xx,a} + p_1 \pi_c \bS_{xx,c} + p_0 \pi_a \bS_{xx,a} \\
&=& \pi_a \bS_{xx,a} + p_1 \pi_c \bS_{xx,c} .
 \end{eqnarray*}
Term $\bC$ is
$
\bC = E\left\{  n^{-1} \sum_{i=1}^n T_i \bX_i \bX_i^\T  \right\}
= p_1 \bS_{xx}.
$
Term $\bD$ is
\begin{eqnarray*}
\bD
&=& E\left\{  \frac{1}{n}\sum_{i=1}^n T_i D_i \bX_i \bX_i^\T \right\}
= E\left\{  \frac{1}{n} \sum_{i=1}^n T_i I_{(U_i = a)} \bX_i \bX_i^\T  + \frac{1}{n} \sum_{i=1}^n T_i I_{(U_i = c)} \bX_i \bX_i^\T \right\} \\
&=& p_1 \pi_a \bS_{xx,a} + p_1 \pi_c \bS_{xx,c}.
\end{eqnarray*}
Term $\bG$ is
\begin{eqnarray*}
\bG
= E\left\{ \frac{1}{n} \sum_{i=1}^n \bX_i Y_i^\obs\right\}
= E\left\{ \frac{1}{n} \sum_{i=1}^n T_i \bX_i Y_i^\obs + \frac{1}{n}   \sum_{i=1}^n (1-T_i) \bX_i Y_i^\obs\right\} 
= p_1 \bS_{x1} + p_0 \bS_{x0}.
\end{eqnarray*}
Term $\bH$ is
$
\bH
= E\left\{ n^{-1} \sum_{i=1}^n T_i \bX_i Y_i^\obs  \right\} 
= p_1 \bS_{x1}.
$
We apply the following formula for the inverse of a block matrix:
$$
\begin{pmatrix}
\bA&\bB\\
\bC&\bD
\end{pmatrix}^{-1} =
\begin{pmatrix}
\bS_\bD^{-1} & -\bA^{-1}\bB\bS_\bA^{-1} \\
-\bD^{-1}\bC\bS_\bD^{-1} & \bS_\bA^{-1}
\end{pmatrix},
$$
where $\bS_\bD = \bA - \bB\bD^{-1}\bC$ and $\bS_\bA = \bD - \bC\bA^{-1} \bB$ are the Schur complements of blocks $\bD$ and $\bA$. Omitting some tedious matrix algebra, we obtain
\begin{eqnarray*}
\bS_\bD =p_0 \pi_c \bS_{xx,c} (\pi_a \bS_{xx,a} + \pi_c \bS_{xx,c})^{-1} \bS_{xx} ,\quad 
\bS_\bA = p_1 p_0 \pi_c \bS_{xx,c},
\end{eqnarray*} 
and the inverse of the block matrix is
$$
\begin{pmatrix}
\bA&\bB\\
\bC&\bD
\end{pmatrix}^{-1} =
\begin{pmatrix}
p_0^{-1} \pi_c^{-1} \bS_{xx}^{-1} (\pi_a \bS_{xx,a} + \pi_c \bS_{xx,c}) \bS_{xx,c}^{-1} &
-p_1^{-1} p_0^{-1} \pi_c^{-1} \bS_{xx}^{-1} (\pi_a \bS_{xx,a} + p_1 \pi_c \bS_{xx,c} )  \bS_{xx,c}^{-1}  \\
 - p_0^{-1} \pi_c^{-1} \bS_{xx,c}^{-1} & p_1^{-1} p_0^{-1} \pi_c^{-1} \bS_{xx,c}^{-1}
\end{pmatrix}.
$$
Therefore, according to \eqref{eq::ABCDGH}, the probability limit of $\widehat{\bgamma}_\TSLS$ is
\begin{eqnarray}
&&p_0^{-1} \pi_c^{-1} \bS_{xx}^{-1} (\pi_a \bS_{xx,a} + \pi_c \bS_{xx,c}) \bS_{xx,c}^{-1}  (p_1 \bS_{x1} + p_0 \bS_{x0}) 
- p_1^{-1} p_0^{-1} \pi_c^{-1} \bS_{xx}^{-1} (\pi_a \bS_{xx,a} + p_1 \pi_c \bS_{xx,c} )  \bS_{xx,c}^{-1}  ( p_1 \bS_{x1} )  \nonumber  \\
&=&
\bS_{xx}^{-1} \bS_{x0} - \pi_a \pi_c^{-1} \bS_{xx}^{-1}\bS_{xx,a} \bS_{xx,c}^{-1} (\bS_{x1} - \bS_{x0} ) \nonumber  \\
&=&
\bgamma_0 - \pi_a \bS_{xx}^{-1} \bS_{xx,a} \bbeta_{c} \equiv \bgamma_\infty , 
 \label{eq::gamma-limit-tsls}
\end{eqnarray} 
and the probability limit of $\widehat{\bbeta}_\TSLS$ is
\begin{eqnarray}
\label{eq::tsls-limit}
- p_0^{-1} \pi_c^{-1} \bS_{xx,c}^{-1}    (p_1 \bS_{x1} + p_0 \bS_{x0})
+p_1^{-1} p_0^{-1} \pi_c^{-1} \bS_{xx,c}^{-1} (p_1 \bS_{x1} )
=\pi_c^{-1} \bS_{xx,c}^{-1} (\bS_{x1} - \bS_{x0}) = \bbeta_c,
\end{eqnarray}
where we use $\bS_{x1} - \bS_{x0} = \pi_c(\bS_{x1,c} - \bS_{x0,c})$, which is guaranteed by exclusion restrictions.

We next find the asymptotic distribution of $\widehat{\bbeta}_\TSLS.$
Following the derivation in Corollary~\ref{coro::ri-nomcompliance}, we first write
\begin{eqnarray*}
\begin{pmatrix}
\widehat{\bgamma}_\TSLS\\
\widehat{\bbeta}_\TSLS
\end{pmatrix}
-
\begin{pmatrix}
\bgamma_{\infty} \\
\bbeta_c
\end{pmatrix}
=
\left\{  \frac{1}{n}   \sum_{i=1}^n
\begin{pmatrix}
\bX_i\\
T_i \bX_i
\end{pmatrix}
(\bX_i^\T, D_i\bX_i^\T)
\right\} ^{-1}
\left\{ \frac{1}{n} \sum_{i=1}^n
\begin{pmatrix}
\bX_i (Y_i^\obs -  \bX_i^\T \bgamma_\infty - D_i \bX_i^\T \bbeta_c )  \\
T_i\bX_i(Y_i^\obs -  \bX_i^\T \bgamma_\infty - D_i \bX_i^\T \bbeta_c )  
\end{pmatrix}
\right\},
\end{eqnarray*}
then introduce
\begin{eqnarray}\label{eq::second-right}
\Delta_\TSLS &=&
\begin{pmatrix}
\bA&\bB\\
\bC&\bD
\end{pmatrix}
^{-1}
\left\{ \frac{1}{n} \sum_{i=1}^n
\begin{pmatrix}
\bX_i (Y_i^\obs -  \bX_i^\T \bgamma_\infty - D_i \bX_i^\T \bbeta_c )  \\
T_i\bX_i(Y_i^\obs -  \bX_i^\T \bgamma_\infty - D_i \bX_i^\T \bbeta_c )  
\end{pmatrix}
\right\} \nonumber \\
&=& 
\begin{pmatrix}
\bA&\bB\\
\bC&\bD
\end{pmatrix}
^{-1}
\begin{pmatrix}
n^{-1}\sumn T_i \bX_i e_i''(1) + n^{-1} \sumn (1-T_i) \bX_i e_i''(0)\\
n^{-1}\sumn T_i \bX_i e_i''(1)
\end{pmatrix} \nonumber \\
&=& 
\begin{pmatrix}
\bA&\bB\\
\bC&\bD
\end{pmatrix}
^{-1}
\begin{pmatrix}
n^{-1}\sumn T_i \bX_i \{ e_i''(1)  - e_i''(0)\} + n^{-1} \sumn \bX_i e_i''(0) \\
n^{-1}\sumn T_i \bX_i e_i''(1)
\end{pmatrix} , \nonumber\\
\label{eq::second-right}
\end{eqnarray}
with $(\bA,\bB,\bC,\bD)$ defined in \eqref{eq::ABCDGH} and $\{ e_i''(1), e_i''(0)\}$ defined in Theorem \ref{thm::TSLS},
and finally recognize that the difference between the above two formulas has high order. 
Again we need only to find the asymptotic distribution of $\Delta_\TSLS$. 
The covariance of the second term on the right hand side of \eqref{eq::second-right} is (dropping the constant sum of $\bX_i e''(0)$) 
\begin{eqnarray*}
&&\cov 
\begin{pmatrix}
n^{-1}\sumn T_i \bX_i \{ e_i''(1)  - e_i''(0)\}  \\
n^{-1}\sumn T_i \bX_i e_i''(1)
\end{pmatrix} \\
&=&
\frac{1}{n^2} \frac{n_1n_0}{n} 
\begin{pmatrix}
\mathcal{S}( \bX \varepsilon ) & \frac{1}{2} [  \S\{ \bX e''(1) \} - \S\{  \bX e''(0) \} + \S( \bX \varepsilon  ) ]\\
 \frac{1}{2} [  \S\{ \bX e''(1) \} - \S\{  \bX e''(0) \} + \S( \bX\varepsilon ) ] & \S\{ \bX e''(1) \}
\end{pmatrix},
\end{eqnarray*}
where the off-diagonal term comes from the finite population covariance between $\bX\{ e''(1) - e''(0) \} $ and $\bX e''(1)$.
Therefore, according to \eqref{eq::second-right}, the asymptotic covariance of $\Delta_\TSLS$ is the $(2,2)$ block of the following matrix
\begin{eqnarray*}
&&\frac{1}{n^2} \frac{n_1n_0}{n} 
\begin{pmatrix}
\bA&\bB\\
\bC&\bD
\end{pmatrix}
^{-1} \cdot \\
&&\begin{pmatrix}
\mathcal{S}( \bX\varepsilon ) & \frac{1}{2} [  \S\{ \bX e''(1) \} - \S\{  \bX e''(0) \} + \S( \bX\varepsilon  ) ]\\
 \frac{1}{2} [  \S\{ \bX e''(1) \} - \S\{  \bX e''(0) \} + \S( \bX\varepsilon ) ] & \S\{ \bX e''(1) \}
\end{pmatrix}\cdot
\begin{pmatrix}
\bA&\bB\\
\bC&\bD
\end{pmatrix}
^{-\T},
\end{eqnarray*}
which is
\begin{eqnarray*}
&&\frac{1}{n^2} \frac{n_1n_0}{n}  \Big\{   
( p_0^{-1} \pi_c^{-1} \bS_{xx,c}^{-1})     \mathcal{S}( \bX\varepsilon  )   ( p_0^{-1} \pi_c^{-1} \bS_{xx,c}^{-1})  ^\T
+ (p_1^{-1} p_0^{-1} \pi_c^{-1} \bS_{xx,c}^{-1})   \S\{ \bX e''(1) \} (p_1^{-1} p_0^{-1} \pi_c^{-1} \bS_{xx,c}^{-1})^\T  \\
&&- ( p_0^{-1} \pi_c^{-1} \bS_{xx,c}^{-1})  [  \S\{ \bX e''(1) \} - \S\{  \bX e''(0) \} + \S( \bX\varepsilon  )  ]  (p_1^{-1} p_0^{-1} \pi_c^{-1} \bS_{xx,c}^{-1})^\T 
\Big\}\\
&=& (\pi_c \bS_{xx,c} ) ^{-1}
\left[
\frac{ \S\{ \bX e''(1) \} }{n_1 }  +  \frac{ \S\{ \bX e''(0) \} }{n_0 } - \frac{ \S(  \bX \varepsilon ) }{n }
\right]
(\pi_c \bS_{xx,c} ) ^{-1} .
\end{eqnarray*}
The asymptotic normality follows from the representation in \eqref{eq::second-right} and the FPCLT. 
\end{proof}

\paragraph{Theorem~\ref{thm::decomposition-noncompliance}: Decomposition of variation in non-compliance.}
The following proof uses two facts: $\tau_a =  \tau_n = 0$, and $ \tau = \pi_c \tau_c.$

\begin{proof}
[Proof of Theorem \ref{thm::decomposition-noncompliance}.]
Write the total treatment effect variation as
\begin{eqnarray*}
S_{\tau\tau} &=&  \frac{1}{n} \sum_{i=1}^n  (\tau_i  - \tau )^2 
= \frac{1}{n} \sum_{i=1}^n \tau_i^2 - \tau^2 \\
&=& \frac{1}{n} \sum_{i=1}^n  I_{(U_i = c) }  \tau_i^2 - \pi_c^2  \tau_c^2 
= \pi_c \left(   \frac{1}{n_c}\sum_{i=1}^n  I_{(U_i = c) }  \tau_i^2 -  \tau_c^2    \right)  + \pi_c(1-\pi_c) \tau_c^2 ,
\end{eqnarray*}
the treatment effect variation explained by compliance status as
\begin{eqnarray*}
S_{\tau\tau, U}  &=& \sum_{u=c,a,n} \pi_u  ( \tau_u - \tau)^2
= \pi_c (\tau_c   - \pi_c \tau_c )^2 + \pi_a (0 -  \pi_c \tau_c)^2 + \pi_n(0- \pi_c  \tau_c)^2\\
&=& \pi_c \tau_c^2  \left\{   (1-\pi_c)^2  + \pi_c(\pi_a + \pi_n)  \right\} 
= \pi_c (1 - \pi_c) \tau_c^2,
\end{eqnarray*}
and the subtotal treatment effect variation for compliers as
\begin{eqnarray*}
S_{\tau\tau,c} = \frac{1}{n_c}   \sum_{i=1}^n I_{(U_i = c) } (\tau_i -  \tau_c)^2
=  \frac{1}{n_c} \sum_{i=1}^n  I_{(U_i = c) }  \tau_i^2 - \tau_c^2 .
\end{eqnarray*}
Therefore, the above three terms has the relationship
$
S_{\tau\tau}  = \pi_c  S_{\tau\tau,c} + S_{\tau\tau, U} . 
$

The decomposition $S_{\tau\tau, c} = S_{\delta \delta, c} + S_{\varepsilon\varepsilon,c}$ follows immediately from the definition of $\bbeta_c.$
\end{proof}

%
%
%\begin{proof}
%[Proof of Corollary \ref{coro::noncompliance}.]
%The proof is similar to the proofs of Theorem \ref{thm::bounds_for_s_tau_tau} and Corollary \ref{coro::independent}.
%\end{proof}

\section{More detailed comments}
\label{appendix::b}

Appendices B.1--B.5 give more details of some technical issues and extensions mentioned in the main text, and Appendix B.6 contains the proofs of the results in \ref{appendix::b}.

\subsection{Covariate adjustment to improve efficiency}

In the main text, the role of covariates has been to model the treatment effect alone. In general, we also want to use covariates to reduce sampling variability of $\widehat{\bbeta}_{\RI}$, just as we can use covariates to get more precise estimates of the average treatment effect.
In particular, the goal is to more precisely estimate $\widehat{\bS}_{xt} \in \mathbb{R}^K$; because these are the only random components in $\widehat{\bbeta}_{\RI}$, if we estimate them more precisely, we estimate $\widehat{\bbeta}_{\RI}$ more precisely as well.
Let $\bW_i \in \mathbb{R}^J$ denote a vector of pretreatment covariates without the intercept term. Because $\bX_i$ and $\bW_i$ have different roles in estimation, they may also contain different sets of covariates, though, in practice, $\bX$ is likely to be a subset of $\bW$.

Following the covariate adjustment approach in survey sampling, we can obtain a model-assisted estimator for $\bbeta$ that uses $\bW$ to reduce sampling variability. To see this, we need several definitions. Define $\overline{\bW} = n^{-1} \sumn \bW_i $ and $\bS_{ww} = n^{-1} \sumn \bW_i \bW_i^\T $, with  $\text{det}(\bS_{ww}) > 0$; define $\overline{\bW}_t$ and $\widehat{\bS}_{ww,t}$ as the sample mean and covariance of $\bW$ under treatment arm $t$; define $\widehat{
\bB}_t \in \mathbb{R}^{J\times K}$ as the regression coefficient of $Y^\obs \bX$ on $\bW$ for treatment arm $t$:
$$
\widehat{\bB}_t = \widehat{\bS}_{ww,t}^{-1} \left\{   \frac{1}{n_t}  \sumn I_{(T_i=t)}  \bW_i (Y_i^\obs \bX_i)^\T    \right\}. 
$$ 
The model-assisted estimator for $\bS_{xt}$ is then 
$$
\widehat{\bS}_{xt}^w = \widehat{\bS}_{xt} - \widehat{\bB}_t^\T  ( \bar{\bW}_t - \bar{\bW}) , \quad (t=0,1).
$$ 
As a result, we can improve the randomization-based estimator by
$$
\widehat{\bbeta}_{\RI}^w = \bS_{xx}^{-1}  (  \widehat{ \bS  }_{x1}^w  - \widehat{ \bS }_{x0}^w  ).
$$

\begin{theorem}
\label{thm::beta-model-assisted}
The model-assisted estimator
$
\widehat{\bbeta}_{\RI}^w 
$
is consistent for $\bbeta$ with asymptotic covariance  
\begin{eqnarray*}
\label{eq::covariance-adjusted}
\bS_{xx}^{-1}  \left[  
   \frac{  \mathcal{S} \{ \bE (1)  \}    }{n_1} 
   + \frac{ \mathcal{S}\{  \bE (0)   \}   }{n_0}  
   - \frac{  \mathcal{S} (  \bDelta  )  }{ n  } 
   \right]   \bS_{xx}^{-1} ,
\end{eqnarray*}
where $\bE_{i}(t) = Y_i(t) \bX_i - \bB_t^\T ( \bW_i - \bar{\bW})$ is the residual term and $\bDelta_i = \bE_i(1) - \bE_i(0)$.
\end{theorem}

The estimator, $\widehat{\bbeta}_{\RI}^w$ uses covariates both to estimate treatment effect variation and to reduce sampling variability. Asymptotically, as long as $\bW$ is predictive of the marginal potential outcomes, the model-assisted estimator will improve precision over the unassisted estimators.

\subsection{Fisherian exact inference}

When $\varepsilon_i=0$ for all $i$, we can obtain exact inference for $\bbeta$ based on the Fisher randomization test \citep{rubin::1980, rosenbaum::2002, ding::2015jrssb}. With a known $\bbeta$, the null hypothesis 
\begin{eqnarray}
H_0(\bbeta): Y_i(1) - Y_i(0) = \bX_i^\T \bbeta \text{ for all } i 
\label{eq::h0-beta}
\end{eqnarray} 
is sharp in the sense of allowing for full imputation of all missing potential outcomes based on the observed data. We can perform randomization test using any sensible test statistic measuring the deviation from the null hypothesis $H_0(\bbeta)$, for example, the test statistic $t(\bT, \bY^\obs;\bbeta)$ can be the difference-in-means, difference-in-medians or the Kolmogorov--Smirnov statistics comparing two samples $\{ Y_i^\obs - \bX_i ^\T \bbeta:T_i=1 , i=1,\ldots, n\}$ and $\{  Y_i^\obs: T_i=0, i=1,\ldots, n\}.$ Then we can obtain a $(1-\alpha)$ level confidence region for $\bbeta$ by inverting a sequence of randomization tests:
$$
\CR_\alpha = 
\{  \bbeta :   \text{Randomization test fails to reject } H_0(\bbeta) \text{ at significance level } \alpha\} .
$$
The confidence region $\CR_\alpha$ is exact regardless of the sample size, and it is valid for general designs of experiments if we use the corresponding assignment mechanism to simulate the null distribution of the test statistic. Due to the duality between testing and interval estimation, we reject $H_0(\bX)$ with $\bbeta_1 = 0$ in Section \ref{sec::omnibus} if $\CR_\alpha \cap \{ \bbeta : \bbeta_1=0 \} $ is an empty set, which controls the type one error rate by $\alpha.$

\subsection{A Variance Ratio Test}

\citet{raudenbush::2015} have noticed that if the variance of the treatment potential outcome is smaller than the control potential outcome, then the correlation between the individual treatment effect and the control potential outcome is negative. This statement does not involve any covariates, but it can be generalized to incorporate systematic and idiosyncratic treatment effect variation. 
Below we give a finite population version of their result.

\begin{theorem}
\label{thm::variance-ratio-theorem}
If the finite population variance of $\{  Y_i(1)  - \bX_i ' \bbeta \}_{i=1}^n$ is smaller than $  \{  Y_i(0) \}_{i=1}^n$, then the idiosyncratic treatment effect variation, $\{ \varepsilon_i \}_{i=1}^n$, is negatively correlated with the control potential outcomes. 
\end{theorem}

Because the condition in Theorem \ref{thm::variance-ratio-theorem} depends only on the marginal distributions of the potential outcomes, we propose a formal variance ratio test of it using the observed data, which is a generalization of a similar theorem in~\citet{ding::2015jrssb}:

\begin{theorem}
\label{thm::variance-ratio-test}
The variance ratio test with rejection region
$$
\frac{  \log s_1^2   - \log s_0^2   }{  \sqrt{    (\widehat{\kappa}_1 - 1)/n_1 + (\widehat{\kappa}_0 - 1)/n_0      }   }
< \Phi^{-1}(\alpha),
$$
has size at least as large as $\alpha$, where $s_1^2$ and $\widehat{\kappa}_1$ are the sample variance and kurtosis of $\{  Y_i^{\obs} - \bX_i^\T  \widehat{\bbeta}_{\RI} : T_i = 1, i=1,\ldots, n\}$, and $s_0^2$ and $\widehat{\kappa}_0$ are the sample variance and kurtosis of $\{Y_i^{\obs}: T_i = 0, i=1,\ldots, n \} $, and $\Phi^{-1}(\alpha)$ is the $\alpha$-th quantile of the standard normal distribution.
\end{theorem}

For finite population inference, the above test in Theorem \ref{thm::variance-ratio-test} is generally conservative, but for superpopulation inference, it is asymptotically exact.

Note that \citet{raudenbush::2015} and Theorem \ref{thm::variance-ratio-theorem} are only about detecting a negative association. Unfortunately, there is no testable condition for a positive association.

\subsection{More on noncompliance: estimating the bounds of the $R^2$s}

The component $S_{\tau\tau,U}$ and and the probability $ \pi_c$ are directly identifiable according to previous discussion. Furthermore, $ S_{\delta\delta,c}$ is also identifiable according to the following result.

\begin{corollary}
\label{coro::complier-systematic}
$S_{\delta\delta,c}$ can be expressed as the expectation of the following quantity:
$$
\frac{1}{\pi_c} \left\{   
\frac{1}{n} \sumn (\delta_i - \tau_c)^2 - \frac{1}{n_1} \sumn T_i(1-D_i)(\delta_i-\tau_c)^2
-\frac{1}{n_0} \sumn (1-T_i)D_i (\delta_i-\tau_c)^2
\right\}.
$$
\end{corollary}

Because $\pi_c$, $\delta_i = \bX_i^\T \bbeta_c$ and $\tau_c$ can be estimated by a plug-in approach, $S_{\delta\delta,c}$ can also be estimated from the observed data.

In the ITT case, estimation of the residual distributions are straightforward. In the noncompliance case, however, we need more discussion about the estimation of $F_{1c}(y)$ and $F_{0c}(y)$, because $U_i$ is a latent variable. To avoid notational clatter, we assume that $\bgamma_{c1}$ and $\bgamma_{c0}$ are known; in practice we can replace them by the randomization-based estimators $\widehat{\bgamma}_{c1,\RI}$ and $\widehat{\bgamma}_{c0,\RI}$, and the consistency of the final estimator will not be affected. Recall the potential residuals $e_i'(1)$ and $e_i'(0)$ defined in \eqref{eq::potential-residual-noncompliance}, and its observed value $e'_i=T_ie'_i(1) + (1-T_i)e'_i(0)$. We define the following quantities
\begin{eqnarray}
\begin{array}{lll}
\widehat{F}_{11}(y) = \frac{1}{n_1} \sumn T_iD_iI_{(e'_i\leq y)},&& 
\widehat{F}_{10}(y) = \frac{1}{n_1} \sumn T_i(1-D_i) I_{(e'_i\leq y)}, \\
\widehat{F}_{01}(y) = \frac{1}{n_0} \sumn (1-T_i) D_i I_{(e'_i\leq y)},&&
\widehat{F}_{00}(y) = \frac{1}{n_0} \sumn (1-T_i)(1-D_i) I_{(e'_i\leq y)}.  
\end{array} 
\label{eq::cdf-c}
\end{eqnarray}

Similar to Corollary \ref{coro::ri-nomcompliance}, we have the following results.

\begin{corollary}\label{coro::cdf-estimation}
For any $y,$
$$
E\{ \widehat{F}_{11}(y) -\widehat{F}_{01}(y)   \} = \pi_c F_{1c}(y),\quad
E\{  \widehat{F}_{00}(y) -\widehat{F}_{10}(y)  \} = \pi_c F_{0c}(y). 
$$
\end{corollary}

Therefore, we can estimate $F_{1c}(y)$ by $\{\widehat{F}_{11}(y) -\widehat{F}_{01}(y)\}/\widehat{\pi}_c$, and estimate $F_{0c}(y)$ by $\{ \widehat{F}_{00}(y) -\widehat{F}_{10}(y) \} / \widehat{\pi}_c$. As we mentioned before, in practice, we use $\widehat{e}_i'$ instead of $e_i'$ in the formulas in \eqref{eq::cdf-c}.

\subsection{Proofs of the theorems and corollaries in Appendix B}

\begin{proof}[Proof of Theorem \ref{thm::beta-model-assisted}.]
The population-level OLS regression matrix of $Y(t) \bX$ onto $\bW$ is
$$
\bB_t = \bS_{ww}^{-1}  \left\{     \frac{1}{n} \sumn   \bW_i \{Y_i(t)\bX_i\}^\T       \right\}  \in \mathbb{R}^{J\times K}.
$$
Define $\widetilde{ \bS  }_{xt}^w = \widehat{ \bS }_{xt} + \bB_t^\T  (\bar{\bW} - \bar{\bW}_t)$ and $\widetilde{ \bbeta }^w_{\RI} = \bS_{xx}^{-1} ( \widetilde{ \bS }_{x1}^w  - \widetilde{ \bS }_{x0}^w ) . $ 
According to the same argument as \eqref{eq::higher-order},
$\widehat{ \bbeta}_{\RI}$ and $\widetilde{ \bbeta }^w_{\RI}$ have the same asymptotic covariance, and in the following we need only to discuss the covariance of $\widetilde{ \bbeta }^w_{\RI}$.
Because
\begin{eqnarray*}
\widetilde{  \bS }_{x1}^w  - \widetilde{  \bS  }_{x0}^w
&=& 
\frac{1}{n_1} \sumn T_i  \left\{   Y_i(1) \bX_i    +  \bB_1^\T  ( \bar{\bW} - \bW_i )       \right\}  
 -  
\frac{1}{n_0}   \sumn (1-T_i) \left\{    Y_i(0)  \bX_i + \bB_0^\T  ( \bar{\bW} - \bW_i  )   \right\} \\
&=& \frac{1}{n_1} \sumn T_i \bE_i(1) - \frac{1}{n_0}   \sumn (1-T_i) \bE_i(0)
\end{eqnarray*} 
can be represented as the difference between the sample means of $\bE_{i}(1)$ and $\bE_{i}(0)$, applying Theorem \ref{thm::randomization-inference} we can obtain its covariance:
\begin{eqnarray*}
\cov\left( \widetilde{  \bS }_{x1}^w  - \widetilde{  \bS  }_{x0}^w \right)  =
   \frac{  \mathcal{S} \{ \bE(1)  \}    }{n_1} 
   + \frac{ \mathcal{S}\{  \bE(0)   \}   }{n_0}  
   - \frac{  \mathcal{S} \{  \bDelta    \}    }{ n},
\end{eqnarray*}
which completes the proof.
\end{proof}

\begin{proof}
[Proof of Theorem \ref{thm::variance-ratio-theorem}.]
For simplicity, we abuse the variance and covariance notation for finite population. For example, $ \var\{  Y(0) \} = \sumn \{ Y_i(0) -\bar{Y}(0)  \}^2/(n-1).$
If $\var\{  Y(1) - \bX ^\T  \bbeta  \} \leq \var\{  Y(0) \}$, then $  \var\{  Y(0)  + \varepsilon  \} \leq \var\{  Y(0) \} .$ Expanding the left hand side,
$$
\var\{  Y(0) \} + \var\{ \varepsilon \} + 2\cov\{  Y(0) , \varepsilon \} \leq \var\{   Y(0) \},
$$
which implies 
$
2\cov\{  Y(0) , \varepsilon \} \leq - \var\{ \varepsilon \} < 0.
$
\end{proof}

Although it is straightforward to prove the conclusion for super population inference of Theorem \ref{thm::variance-ratio-test} by using \citet[][Theorem 2, Supplementary Material]{ding::2015jrssb} and Slutsky's Theorem, it is less obvious to prove the conclusion for finite population inference. To simplify the proof, we first prove the following lemma. 
Let $(c_1, \cdots, c_n)^\T $ and $(d_1, \ldots, d_n)^\T  $ be two vectors of nonnegative constants with the same mean $m>0$ but different variances $S_{cc}$ and $S_{dd}$. The difference vector $(c_1-d_1, \ldots, c_n - d_n)^\T $ has mean zero and variance $S_{c-d,c-d}$.  Let 
$$
\widehat{\theta}_c = \frac{1}{n_1} \sum_{i=1}^n T_i c_i , \quad 
\widehat{\theta}_d = \frac{1}{n_0} \sum_{i=1}^{n} (1-T_i) d_i 
$$ 
be two sample means of the treatment and control group, respectively.

\begin{lemma}
\label{lemma::ratio-asymptotics}
Under the regularity conditions for the FPCLT, $\log \widehat{\theta}_c - \log \widehat{\theta}_d$ has asymptotic mean zero and variance
\begin{eqnarray}
\frac{1}{m^2} \left(   
\frac{S_{cc}}{n_1} + \frac{S_{dd}}{n_0} - \frac{S_{c-d,c-d}}{n}
\right).
\label{eq::variance-log-diff}
\end{eqnarray}
\end{lemma}

\begin{proof}
[Proof of Lemma \ref{lemma::ratio-asymptotics}.]
According to the FPCLT, we have the following joint asymptotic normality of $\widehat{\theta}_c$ and $\widehat{\theta}_d$:
$$
\begin{pmatrix}
\widehat{\theta}_c \\
\widehat{\theta}_d
\end{pmatrix}
=\begin{pmatrix}
n_1^{-1} \sum_{i=1}^n   T_i c_i \\
n_0^{-1} \sum_{i=1}^n   (1-T_i)d_i
\end{pmatrix}
\stackrel{a}{\sim} N
\left[
\begin{pmatrix}
m\\
m
\end{pmatrix} ,
\begin{pmatrix}
V_{cc} & V_{cd}\\
V_{cd} & V_{dd}
\end{pmatrix}
\right],
$$
where 
$$
V_{cc} = \frac{n_0}{n_1n} S_{cc},\quad
V_{dd} = \frac{n_1}{n_0n} S_{dd},\quad 
V_{cd}=-\frac{1}{2n} (S_{cc}+S_{dd} -S_{c-d,c-d}).
$$
Applying Taylor expansion at $m$, we have $\log \widehat{\theta}_c - \log\widehat{\theta}_d = \{ (\widehat{\theta}_c-m) - (\widehat{\theta}_d-m) \} / m + o_P(n^{-1/2})$, which, coupled with \citet{neyman::1923}'s variance formula, gives the asymptotic variance of $\log \widehat{\theta}_c - \log\widehat{\theta}_d$ in (\ref{eq::variance-log-diff}).
\end{proof}

\begin{proof}
[Proof of Theorem \ref{thm::variance-ratio-test}.]
First, as a direct consequence of Lemma \ref{lemma::ratio-asymptotics}, the finite sample variance is always larger than the super population variance, unless $S_{c-d,c-d}=0$. Therefore, we need only to show that the test in Theorem \ref{thm::variance-ratio-test} is asymptotically exact for super population inference, and the asymptotic size of the test is no larger than $\alpha$ for finite population inference. 

Second, replacing $\bbeta$ by its consistent estimator $\widehat{\bbeta}_{\RI}$ does not affect the asymptotic distribution of the test statistic, due to Slutsky's Theorem. For simplicity, we treat $\bbeta$ as known in our asymptotic analysis.

With the two ingredients above, Theorem \ref{thm::variance-ratio-test} follows directly from the variance ratio test in \citet[][Theorem 2, Supplementary Material]{ding::2015jrssb}.
\end{proof}

\begin{proof}
[Proof of Corollary \ref{coro::complier-systematic}.]
The conclusion follows from
\begin{eqnarray*}
E\left\{  \frac{1}{n_1} \sumn T_i(1-D_i)(\delta_i-\tau_c)^2 \right\} &=&
E\left\{  \frac{1}{n_1} \sumn T_iI_{(U_i=n)}(\delta_i-\tau_c)^2 \right\}
= \frac{1}{n} \sumn I_{(U_i=n)}(\delta_i-\tau_c)^2,\\
E\left\{  \frac{1}{n_0} \sumn (1-T_i)D_i (\delta_i-\tau_c)^2 \right\}&=&
E\left\{  \frac{1}{n_0} \sumn (1-T_i) I_{(U_i=a)} (\delta_i-\tau_c)^2 \right\}
=  \frac{1}{n} \sumn I_{(U_i=a)} (\delta_i-\tau_c)^2 .
\end{eqnarray*}
\end{proof}

\begin{proof}
[Proof of Corollary \ref{coro::cdf-estimation}.]
We rewrite
\begin{eqnarray*}
\widehat{F}_{11}(y) &=& % \frac{1}{n_1} \sumn T_iD_iI_{(e'_i\leq y)}  = 
 \frac{1}{n_1} \sumn T_i I_{(U_i=c)} I_{ \{ e_i(1)\leq y  \} } + \frac{1}{n_1} \sumn T_i I_{(U_i=a)} I_{ \{ e_i(1)\leq y  \} } ,\\
\widehat{F}_{10}(y) &=& % \frac{1}{n_1} \sumn T_i(1-D_i) I_{(e'_i\leq y)} = 
 \frac{1}{n_1}  \sumn T_i I_{(U_i = n)}  I_{ \{ e_i(1) \leq y  \}  } , \\
\widehat{F}_{01}(y) &=& % \frac{1}{n_0} \sumn (1-T_i) D_i I_{(e'_i\leq y)} = 
 \frac{1}{n_0} \sumn (1-T_i) I_{(U_i=a)} I_{  \{  e_i(0)\leq y  \}} ,\\
\widehat{F}_{00}(y) &=& % \frac{1}{n_0} \sumn (1-T_i)(1-D_i) I_{(e'_i\leq y)} = 
 \frac{1}{n_0} \sumn (1-T_i) I_{(U_i= c)} I_{  \{ e_i(0)\leq y \}  } + \frac{1}{n_0} \sumn (1-T_i) I_{(U_i=n)} I_{  \{  e_i(0)\leq y  \} } .
\end{eqnarray*}
In the above formulas, the random components are the $T_i$'s, and therefore, the corollary follows from Lemma \ref{lemma::cov} and the linearity of expectations.
\end{proof}

\end{document}